\documentclass[reqno,english]{amsart} 
\usepackage{mathtools}
\usepackage{enumerate}  
\usepackage{upgreek}
\usepackage{amsthm}
\usepackage{bbm}
\usepackage{bm}

\usepackage[capitalize]{cleveref}
\title{Critical Window of The Symmetric Perceptron}
\author{Dylan J. Altschuler}
\address{D.J.\ Altschuler\hfill\break
Courant Institute\\ New York University\\
251 Mercer Street\\ New York, NY 10012, USA.}
\email{dylan.altschuler@courant.nyu.edu}
\newcommand*{\PP}[2][]{\mathbb{P}_{#1}\left[#2\right]}

\newcommand*{\E}[2][]{\mathbb{E}_{#1\unskip\space}\left[#2\right]}

\newcommand*{\cO}[1]{\mathcal{O}\left(#1\right)}

\newcommand*{\cOm}[1]{\Omega\left(#1\right)}

\newcommand*{\cT}[1]{\Theta\left(#1\right)}
\newcommand*{\cE}[1]{\exp\left\{#1\right\}}
\newcommand*{\oo}{\mathrm{o}}
\newcommand*{\mO}[1]{\left(1 + \mathcal{O}\left(#1\right)\right)}

\newcommand{\disc}{\mathrm{disc}}

\newcommand*{\ZZ}{\mathbb{Z}}
\newcommand*{\RR}{\mathbb{R}}
\newcommand*{\NN}{\mathbb{N}}

\newcommand*{\ind}{\mathbbm{1}}

\newcommand*{\la}{\langle}
\newcommand*{\ra}{\rangle}
\newcommand*{\eps}{\epsilon}

\newcommand*{\del}{\delta}

\newcommand{\defeq}{\coloneqq}

\newcommand*{\pa}[1]{\left(#1\right)}

\newcommand*{\cb}[1]{\left\{#1\right\}}

\newcommand*{\innerprod}[2]{\left\la #1, #2 \right\ra}
\newcommand*{\ba}[1]{\left|#1\right|}

\newcommand*{\cL}{\mathcal{L}}
\newcommand*{\kc}{{K_c}}

\newcommand*{\cp}{\bm{\alpha_*}}
\newcommand*{\ac}{{\alpha_c}}

\newcommand*{\I}{\mathrm{I}}
\newcommand*{\creg}{C_{r}}

\newtheorem{theorem}{Theorem}[section]
\newtheorem{conjecture}[theorem]{Conjecture}
\newtheorem{corollary}[theorem]{Corollary}
\newtheorem{definition}[theorem]{Definition}
\newtheorem{lemma}[theorem]{Lemma}
\newtheorem{proposition}[theorem]{Proposition}

\begin{document}
\maketitle
\begin{abstract}    
We study the critical window of the symmetric binary perceptron, or equivalently, random combinatorial discrepancy. Consider the problem of finding a $\pm1$-valued vector $\sigma$ satisfying $\|A\sigma\|_\infty \le K$, where $A$ is an $\alpha n \times n$ matrix with iid Gaussian entries. For fixed $K$, at which constraint densities $\alpha$ is this constraint satisfaction problem (CSP) satisfiable? A sharp threshold was recently established by Perkins and Xu \cite{PX}, and Abbe, Li, and Sly \cite{ALS1}, answering this to first order. Namely, for each $K$ there exists an explicit critical density $\ac$ so that for any fixed $\eps > 0$, with high probability the CSP is satisfiable  for $\alpha n < (\ac - \eps ) n$ and unsatisfiable for $\alpha n > (\ac + \eps) n$. This corresponds to a bound of $o(n)$ on the size of the critical window. 
    
We sharpen these results significantly, as well as provide exponential tail bounds. Our main result is that, perhaps surprisingly, the critical window is actually at most of order $\log(n)$. More precisely, for a large constant $C$, with high probability the CSP is satisfiable for $\alpha n < \ac n - C\log(n)$ and unsatisfiable for $\alpha n > \ac n + C$. These results add the the symmetric perceptron to the short list of CSP models for which a critical window is rigorously known, and to the even shorter list for which this window is known to have nearly constant width.
\end{abstract}

\section{Introduction}
Discrepancy arises as a fundamental quantity in combinatorics, functional analysis, geometry, and optimization. The \textbf{combinatorial discrepancy} of a matrix $A \in \RR^{\alpha n \times n}$ is given by the following optimization problem:
\[
    \mathrm{disc}(A) \defeq  \min_{\sigma \in \cb{-1,+1}^n} \|A\sigma\|_\infty\,.
\]
In words, discrepancy is the task of finding some $\sigma \in \cb{\pm 1}^n$ that has small inner product in absolute value with all rows of $A$. This task is naturally encoded as a CSP (constraint satisfaction problem). Fix a positive parameter $K$ and look for assignments $\sigma \in \cb{\pm 1}^n$ that satisfy each of the $\alpha n$ constraints encoded by the rows of $A$:
\begin{equation}\label{eq:CSP}
    \cb{\ba{\innerprod{A_1}{\sigma}} \le K} \wedge \dots \wedge \cb{\ba{\innerprod{A_{\alpha n}}{\sigma}} \le K} \,.
\end{equation}
We say the problem is satisfiable if such an assignment $\sigma$ exists. For fixed $K$, if we construct \cref{eq:CSP} by adding one constraint at a time, how many constraints can we add before the equation becomes unsatisfiable? 

Much recent work focuses on understanding, for random $A$, the parameters $K$ and $\alpha$ for which \cref{eq:CSP} is satisfiable with positive probability. For random CSP, it is natural to suspect the existence of a sharp threshold \cite{DSS-KSAT, fried,hatami}. A \textit{sharp threshold} is said to occur if there exists some $\ac$ so that for any $\eps > 0$, the probability that \cref{eq:CSP} is satisfiable tends to one for $\alpha n < (\ac -\eps) n$ and to zero for $\alpha n > (\ac + \eps) n$, as $n$ grows. We now formally state the model of randomness that will be studied.

\begin{definition}\label{def:ensemble}
Let $\alpha \defeq  \alpha(n)$. Say an $\alpha n \times n$ matrix $A$ is $(\alpha,n)$-Gaussian if it has iid normal entries with mean 0 and variance $1/n$.
\end{definition}

The normalization is so that $\disc(A)$ is on the constant scale.
We will only consider $\alpha \defeq  \alpha(n)$ tending to a finite, strictly postive limit (see \cref{section:prev-works} for $\alpha \to 0$). The discrepancy of such matrices has gathered significant recent attention in the statistical physics community due to connections with the celebrated \textit{Binary Perceptron}. 

The binary perceptron is an idealized model of learning that has enjoyed almost six decades of intense study by the statistical physics community. While there is an extensive body of detailed physics predictions for the perceptron, these predictions have largely resisted rigorous mathematical treatment \cite{dingsun-perceptron, gd1,mez, rb, tgd, small-densities}. Aubin, Perkins, and Zdeborov\'a \cite{APZ} had the insight of interpreting combinatorial discrepancy as a ``symmetric'' analogue of the binary perceptron. Happily, discrepancy is more amenable to analysis and is conjectured to capture most of the interesting behaviours of the perceptron. This connection has already lead to several breakthroughs in both problems \cite{ALS2,ALS1, gamarnik-sbp,nakajima-sun, PX}. 

In brief, the reduction between the SBP and discrepancy is as follows. We are interested in the set of parameters $(K,\alpha)$ for which \cref{eq:CSP} is satisfiable. One could equally well fix $\alpha$ and ask for the smallest $K$, or else fix $K$ and ask for the largest $\alpha$. The former is adopted in discrepancy and the latter for the SBP. Formally, 

\begin{definition}[Symmetric perceptron]
    For a fixed parameter $K$, an instance of the associated \textbf{symmetric binary perceptron} (SBP) problem is given by a sequence of iid Gaussian vectors $\{a_i\}_{i \in \NN}$ with 
    \[
        a_i \sim N(0,n^{-1}\I_{n \times n})\,.
    \]
    The \textbf{storage capacity} of the SBP is the random variable $\cp$ defined as the largest $\alpha$ so that $\alpha n \in \ZZ$ and there exists $\sigma \in \cb{\pm 1}^n$ with 
    \[
        \max_{i \in [\alpha n]} \ba{\innerprod{a_i}{\sigma}} \le K \,.
    \]
\end{definition} 
The original binary perceptron model asks for $\sigma$ to have positive, rather than small, inner product with rows. A simple but important observation (made formal in \cref{lemma:alpha-K} below) is the following.

\begin{quote}
    \textit{Let $\kc$ and $\ac$ denote the ``typical values'' of $\disc(A)$ and $\cp$ respectively. A bound on $|\disc(A)-\kc|$ is equivalent to a bound on $|\cp-\ac|$ of the same order. Thus, fluctuation bounds on the capacity of the SBP or the discrepancy of a Gaussian matrix are equivalent.}
\end{quote}

As such, we discuss these quantities interchangeably. For matrices with iid Gaussian or Rademacher entries, the existence of a \textit{coarse threshold} was shown by Aubin, Perkins, and Zdeborov\'a \cite{APZ} via the second moment method. They also noted the second moment method fails to give a sharp threshold and conjectured a sharp threshold should still occur. The failure of the second moment method to yield a sharp threshold---usually a difficult obstacle to overcome---is a key feature of many interesting CSP, including the binary perceptron. (In fact, the second moment method famously fails to even yield a coarse threshold for the binary perceptron). In an impressive series of works, Perkins and Xu \cite{PX} and Abbe, Li, and Sly \cite{ALS1} independently established the sharp threshold for the SBP, resolving the conjecture of \cite{APZ}.

When a random CSP admits a sharp threshold, a next natural question is to identify the width of the critical window, sometimes also called the scaling window. Namely, fix some small constant $\del > 0$; as a function of $n$, how many constraints must be added for the probability that \eqref{eq:CSP} is satisfiable to decrease from $1-\delta$ to $\delta$? The following proposition summarizes what is known:

\begin{proposition}[Sharp threshold \cite{PX,ALS1}]\label{prop:state}
For any positive constants $K$, $\eps$, and  $\delta$, there exists $\ac \defeq  \ac(K)$ (given explicitly in \cref{eq:ac-def}) so that for $A$ drawn from the $(\alpha,n)$-Gaussian ensemble and $n$ sufficiently large, 
\begin{align*}
        \alpha n < \ac n - \eps n &\implies \PP{\disc(A) > K} < \delta\,, \\
        \alpha n > \ac n + \eps n &\implies \PP{\disc(A) < K} < \delta\,.
    \end{align*}
Equivalently, the capacity of the symmetric perceptron satisfies
    \[
        \PP{\ac n - \eps n < \cp< \ac n + \eps n } > 1 - 2\del
    \]
\end{proposition}
No further quantification has been established or even conjectured. Our main contribution is to narrow the current bound on the critical window from $o(n)$ to $\cO{\log(n)}$. Strikingly, the list of CSP for which the critical window is even known to be polynomially smaller than $n$, let alone constant, is very short. This is discussed further in \cref{section:prev-works} below. We begin with a simplified, qualitative statement of our main result:

\begin{theorem}[\textbf{Main result:} logarithmic critical window]\label{thm:window} 
    For any positive constants $K$ and $\delta$, there exists an explicit constant $\ac \defeq  \ac(K)$ and some sufficiently large $c > 0$ so that for $A$ drawn from the $(\alpha,n)$-Gaussian ensemble and $n$ sufficiently large, 
    \begin{align*}
        \alpha n < \ac n - c\log(n) &\implies \PP{\disc(A) > K} < \delta\,, \\
        \alpha n > \ac n + c &\implies \PP{\disc(A) < K} < \delta\,.
    \end{align*}
    Equivalently, the capacity of the symmetric perceptron satisfies
    \[
        \PP{\ac n - c < \cp < \ac n + c \log(n)} > 1 - 2\del
    \]
\end{theorem}

We actually prove a significant quantitative strengthening of this. Our main technical contribution gives sub-exponential tails on the fluctuations of discrepancy and directly implies \cref{thm:window}. 

\begin{theorem}[\textbf{Fluctuations}: capacity]\label{thm:main}
    Fix $K > 0$, and let $\cp$ be the storage capacity of the corresponding SBP. There exist positive constants $c$ and $\ac := \ac(K)$ depending only on $K$ so that for any sufficiently large fixed $x$, the following holds for all sufficiently large $n$:
    \[
        \PP{\ba{\cp - \ac} > \frac{x \log (n)}{n}} \le \cE{-c x\log\pa{n}}\,.
    \]
    In addition, there exist some other positive constants $c$ and $C$ depending only on $K$ so that for all $n$ sufficiently large,
    \[
        \frac{c}{n} \le \mathrm{Var}(\cp)^{1/2} \le \frac{C}{n}\log (n) \,.
    \]
\end{theorem}

\cref{thm:main} implies \cref{thm:window} as an immediate corollary, as well as an analogous bound on discrepancy:

\begin{theorem}[\textbf{Fluctuations}: discrepancy]\label{thm:main-disc}
    Fix $\alpha > 0$ and let $A$ be from the $(\alpha, n)$-Gaussian ensemble. There exist positive constants $c$ and $\kc := \kc(\alpha)$ depending only on $\alpha$ so that for any sufficiently large fixed $x$, the following holds for all sufficiently large $n$:
    \begin{equation}\label{eq:main-tail}
        \PP{\ba{\disc(A) - \kc} > \frac{x \log (n)}{n}} \le \cE{-c x\log\pa{n}} \,.
    \end{equation}
    In addition, there exist some other positive constants $c$ and $C$ depending only on $\alpha$ so that for all $n$ sufficiently large,
    \begin{equation}\label{eq:main-var}
        \frac{c}{n} \le \mathrm{Var}(\disc(A))^{1/2} \le \frac{C}{n}\log (n) \,.
    \end{equation}
\end{theorem} 
Some remarks:
\begin{enumerate}
    \item Existing general tools fall quite short of being able to capture when the critical window of a CSP has width polynomially smaller than $n$---let alone a window of constant width. See e.g. a result of Xu giving a sharp threshold for the (asymmetric) perceptron \cite{xu} based on Hatami's theorem \cite{hatami}, which establishes a sharp threshold sequence of width at most $n /\pa{\log(\log(n))}^{1/10}$.
    \item It is natural to compare discrepancy to the smallest singular value of a Gaussian matrix, since they solve somewhat similar minimization problems. \cref{thm:main} shows discrepancy is significantly more concentrated, with fluctuations of order $n^{-1}$, whereas the smallest singular value of a Gaussian matrix follows the Tracy-Widom law with fluctuations of order $n^{-2/3}$ \cite{RV-ssv}.
    \item Our lower-bound on the variance shows that the critical window is at least constant. C.f. random linear equations over a finite field, which can be written as a CSP with a degenerate window \cite{sandpile,dreg}. By ``degenerate'', we mean that the critical window has size exactly one, as well as a deterministic location; this corresponds to random (regular) equations being solvable w.h.p. exactly until the number of equations $\alpha n$ exceeds the number of variables $n$. 
\end{enumerate}

We conclude by highlighting a key intermediate result. Let us briefly introduce a standard perspective for studying CSP: counting solutions. Fix $\alpha$ and define $Z_K \defeq  Z_{K,\alpha}(A)$ to be the set of solutions incurring discrepancy at most $K$.
\[
    Z_{K} \defeq  \cb{x \in \cb{\pm 1}^n:~\|Ax\|_\infty \le K}, \quad \disc(A) = \inf\{K:~Z_K \neq \emptyset\} \,.
\]
A natural prediction is that $\disc(A)$ should concentrate around $\kc \defeq  \kc(\alpha)$, defined as the smallest $K$ for which $\E{|Z_K|} \ge 1$. Showing $|Z_K|$ concentrates around its mean would verify this prediction. Markov's inequality readily implies $\disc(A) > (\kc -\eps)$ with exponentially high probability, for any $\eps > 0$.  The coarse threshold of Aubin, Perkins, and Zdeborov\'a \cite{APZ}, which states $\disc(A) < (\kc + \eps)$ with positive probability for any $\eps > 0$, corresponds to a variance bound:
\begin{equation}\label{eq:2mm-loose}
    \forall \eps > 0, \quad 1 < c \le \frac{\E{|Z_{K_c + \eps}|^2}}{\E{|Z_{K_c + \eps}|}^2} \le C < \infty\,.
\end{equation}
While \cref{eq:2mm-loose} implies that the variance of $|Z_K|$ is small enough to obtain a coarse threshold, it also implies the variance of $|Z_K|$ is too large to obtain a sharp threshold. Rather, it is $\log |Z_K|$ that concentrates well. We will soon check that for any small positive constant $\eps$, $\log \E{|Z_{\kc + \eps}|} = \cT{\eps n}$. The sharp threshold \cref{prop:state} of \cite{PX,ALS1} corresponds to the asymptotic bound that for any positive $\delta$, there exists large $M$ with
\begin{equation}\label{eq:loose}
    \lim_{n \to \infty} \PP{ \ba{ \log|Z_{\kc + \eps}| - \log \E{|Z_{\kc + \eps}|} } > M \log(n)} \le \delta\,.
\end{equation}
Our contribution is to give a non-asymptotic sharpening of \cref{eq:loose} in which $\eps$ and $\delta$ have explicit, near-optimal dependence on $n$. As a necessary intermediate step, we show a significant strengthening of \cref{eq:2mm-loose}. Namely, the second moment method yields a non-trivial bound even \textit{exactly} down to the critical threshold $\kc$. This is new and perhaps surprising.

\begin{theorem}[Coarse threshold, up to criticality]\label{thm:2mm-main} 
    Fix any $\alpha >0$ and let $\kc := \kc(\alpha)$. There exists $C\defeq  C(\alpha) >0$ so that for all $n$ sufficiently large,
    \begin{equation}
    \frac{\E{|Z_{K_c,\alpha}|^2}}{\E{|Z_{K_c,\alpha}|}^2} \ge C > 0\,.
    \end{equation}
    Consequently,
    \begin{equation}
        \PP{\disc(A) \le \kc} > C \,.
    \end{equation}
\end{theorem}
Since the second moment method can be carried out all the way to criticality, it is natural to suspect similar improvements are possible in \cref{eq:loose}. Indeed, this is the content of \cref{thm:main-disc}. We choose to refine the techniques of \cite{PX} because they naturally lead to strong tail bounds, albeit at the cost of an extra logarithm for the width of the critical window. It seems likely that the techniques of \cite{ALS1} can be used to show $\disc(A) = \kc + \cT{1/n}$, resolving this extra log. However, the techniques of \cite{ALS1} are asymptotic and do not easily lead to tail bounds. 

Simultaneously providing strong tail bounds and capturing the exact size of fluctuations is left as an open question; we conjecture that our main result can be improved both by removing the logarithmic mismatch in \cref{eq:main-var}, and also by giving subgaussian rather than subexponential tails in \cref{eq:main-tail}. 
\begin{conjecture}
    $|\disc(A)-\kc|$ is $1/n^2$-subgaussian. 
\end{conjecture}
\textit{Author's note: after the appearance of this work, Sah and Sawhney \cite{ss} showed the lower tails of discrepancy are actually sub-exponential. For applications in discrepancy theory, analysis of the upper tails is essential and remains open.}

\subsection{Notation} Throughout this work, we use $c$ and $C$ to denote positive universal constants that may take different values on different lines. Any important constants will be distinguished with subscripts. For functions (namely $p$, $q$, $F$, and $Z$, all defined in the following section) that depend on the parameters $n$, $K$, and $\alpha$, we frequently suppress some of these dependencies when there is no ambiguity, particularly when one of these parameters is held fixed. \\

\section{Previous Works}\label{section:prev-works}

\textbf{Critical windows of random CSP.} 
The previous sharp threshold results \cite{ALS1} and \cite{PX} on the symmetric perceptron correspond to a bound of $o(n)$ on the width of the critical window. In the literature of random CSP, it is rare to find optimal, or even non-trivial, improvements over $o(n)$ on the critical window. General tools from Fourier Analysis such as the sharp threshold sequence theorems of Friedgut and Bourgain \cite{fried} and Hatami \cite{hatami}, are the best known bounds for most CSP (the perceptron \cite{xu}, k-SAT \cite{SAT}, and k-coloring \cite{kcolor}, etc.). Even when applicable, the Friedgut-Bourgain sharp threshold theorem gives only the qualitative upper-bound $o(n)$, and Hatami's theorem gives a slightly improved $n/\log(\log(n))$. It is possible that many of these standard CSP should have critical windows {\em polynomially} smaller than $n$.  

We are only aware of two other CSP for which such an improvement has been shown:  ``constrained'' XOR-SAT \cite{XOR}, and 2-SAT \cite{2SAT}. Both enjoy very special structure: XOR-SAT is essentially a ``matrix invertibility'' problem, and its critical window is at most constant due to dimension-counting arguments. The 2-SAT model has critical window of size $\cT{n^{2/3}}$, and the analysis employs from a series of clever reductions to a well-studied percolation problem.

Let us also clarify the difference between our result and sharp threshold results for ``regular'' models, e.g. random regular NAE-SAT. A random regular constraint satisfaction problem on $n$ variables is a collection of clauses each involving $k$ variables, with each variable in $d$ clauses. The number of clauses must then be 
\[
nd/k\,.
\]
Ding, Sly, and Sun showed for random regular NAE-SAT \cite{NAE-SAT} that the largest value of $d$ for which there are solutions concentrates around an explicit constant, with at most constant fluctuations. Since a constant change in $d$ results in a change of order $n$ to the number of clauses, this kind of sharp threshold for $d$ seems closer to a qualitative sharp threshold rather than the type of critical window bound investigated here.

Finally, it is quite special that the symmetric perceptron enjoys such a small critical window: there are many sources of fluctuations that could preclude a CSP from having a constant-size critical window. For example, Wilson showed the critical window of the infamous K-SAT problem is at least $\cOm{\sqrt{n}}$. His argument proceeds by noting that the number of trivially satisfiable clauses is of order $n$ with CLT-like fluctuations of order $\cT{\sqrt{n}}$. Another example is the critical window of the binary perceptron, which is conjectured by Ding and Sun \cite{dingsun-perceptron} to be of order $\cT{\sqrt{n}}$ . They posit that such large fluctuations should occur due to fluctuations in the barycenter of the solution set.  \\

\textbf{Coarse threshold.}
The connection between discrepancy and the perceptron model was introduced by Aubin, Perkins and Zdeborov\'a in 2019 \cite{APZ}. They established a \textit{coarse threshold} via the a second moment method. More precisely, they showed---conditional on a numerical conjecture (see \cref{section:UB-1})---for any $\eps > 0$, that $\disc(A) > \kc-\eps$ with high probability and $\disc(A) < \kc-\eps$ with positive probability. Their work also establishes the negative result that the second moment method does not suffice to prove a sharp threshold, but they nonetheless conjectured a sharp threshold (i.e. the upper-bound should hold with high, rather than positive, probability) . 

This conjecture was proven by Abbe, Li, and Sly \cite{ALS1}, and Perkins and Xu \cite{PX} independently and simultaneously. \\

\textbf{Sharp threshold (cycle counting).} Abbe, Li, and Sly used a ``dense graph'' analogue of the so-called small subgraph conditioning technique of Robinson and Wormald \cite{rw}. They established both the sharp threshold as well as an explicit asymptotic description of the number of solutions. The small subgraph conditioning method is a celebrated technique to correct the failure of the second moment method for counting objects in sparse random graphs. The idea is to explain the extra variance by explicitly quantifying how the number of cycles in a graph affects this count. 

Interestingly, the techniques of \cite{ALS1} have a strong analogue in the literature of the SK (Sherrington-Kirkpatrick) model. A seminal paper of Aizenman, Lebowitz, and Ruelle \cite{aiz}---written in 1987, actually predating the appearance of the small subgraph conditioning technique---established the asymptotic log-normal limit for the fluctuations in the free energy of the SK model at high temperatures, also using cycle counts in a dense graph, with many of the same objects appearing. The connection between the techniques of \cite{aiz} and \cite{rw} were further investigated and made explicit in a recent work of Banerjee \cite{banerjee} that establishes a log-normal limit for a perturbation of the SK model. \\

\textbf{Sharp threshold (cavity method).} Perkins and Xu independently and simultaneously resolved the sharp threshold conjecture for the SBP with different techniques. They adapted and refined an argument of Talagrand \cite{talagrand1999intersecting} originally used for the binary perceptron. The technique is a ``rigorous cavity method'', in which martingale techniques are used to show concentration of the log-partition function as rows of the matrix are added. We discuss this method further in \cref{section:UB-ST}. 

Sun and Nakajima \cite{nakajima-sun} recently extended the method of Talagrand to give a sharp threshold {\em sequence} for wide class of random matrices, yielding comparable results to \cite{ALS1,PX} in a broader setting, notably including the binary perceptron. (A sharp threshold sequence is a weaker notion of a sharp threshold in which the capacity concentrates around some $\ac(n)$ for each $n$, but $\pa{\ac(n)}_{n=1}^\infty$ may not converge to some value that is independent of $n$.) \\

\textbf{Algorithms.} The foundational result of combinatorial discrepancy is Spencer's ``Six standard deviations suffice'' \cite{spencer1985six}, which asserts for \textit{any} square matrix $A$ with entries in $[-n^{-1/2},n^{1/2}]$, deterministically $\disc(A) \le 6$. Finding an efficient algorithm that achieves Spencer's bound took almost 25 years. While a number of such algorithms were discovered in the last decade, and while many of these easily generalize beyond bounded entries to our setting of iid random entries (e.g. \cite{int-feas}), they are all far from achieving the optimal constant $\kc$. There are intimate connections between average-case complexity and the geometry of the ``solution space'' for optimization problems that heavily motivated the recent work on sharp thresholds for the perceptron. We refer the interested reader to \cite{APZ,ALS1,ALS2, PX,gamarnik-sbp} for further discussion. \\

\textbf{Wide matrices.} The discrepancy of $\alpha n \times n$ matrices  with $\alpha \to 0$ has applications in a surprising range of combinatorial optimization problems. The case of matrices with iid Bernoulli$(p)$ entries is of particular interest due to practical reasons as well as connections with the long-standing Beck-Fiala conjecture. The second moment method readily gives a sharp threshold for Gaussian matrices \cite{turner2020balancing}. However for Bernoulli matrices, there are fundamental obstacles to analyzing the second moment if $p \defeq  p(n)$ vanishes quickly. A sharp characterization of discrepancy was first established for large values of $p$ \cite{hoberg-rothvoss,potukuchi2018discrepancy, bansal-meka,ezra-lovett,franks2020discrepancy, macrury2021phase}, and eventually for all $p$ \cite{dja-jnw}. 

\section{Preliminaries} \label{section:UB-1}

\subsection{First Moment} Let us adopt the language of random Constraint Satisfaction Problems (CSP's). For an $(\alpha,n)$-Gaussian matrix $A$, define the sublevel set $Z$ by:
\[
    Z_{K,\alpha} \defeq  \cb{x \in \cb{\pm 1}^n:~\|Ax\|_\infty \le K},
\]
so that
\[
    \disc(A) = \inf\{K:~Z_{K,\alpha} \neq \emptyset\} \,.
\]
Then $Z_{K,\alpha}$ is the set of solutions to the ``discrepancy instance'' with parameters $K$ and $\alpha$ given by the matrix $A$. The matrix $A$ encodes a set of constraints: for each row $a_i$, a solution $x$ needs to satisfy $|\innerprod{x}{a_i}| \le K$. It is natural to suspect that for fixed  $\alpha$, the probability of $Z_{K,\alpha}$ being empty undergoes a rapid transition from zero to one as $K$ increases. 

A standard first approach to understanding the value of an optimization problem is the first moment method. If $\E{|Z_{K,\alpha}|}$ is vanishing, then by Markov's inequality there are no solutions with high probability. This provides a lower bound on $\disc(A)$. For an $(\alpha,n)$-Gaussian matrix $A$, it is an elementary computation that for $W$ a standard normal,
\begin{align*}
    \frac{1}{n}\log \E{\ba{ Z_{K,\alpha} }} &= \log(2) +  \alpha \log p_K\,, \\
    p_K \defeq  \PP{|W| \le K} &=  \frac{1}{\sqrt{2\pi}}\int_{-K}^{K}e^{-w^2/2}dw\,.
\end{align*}
For fixed $K$, we define $\ac(K)$ the ``critical value'' of $\alpha$ so that the expected number of solutions is one: 
\begin{equation}\label{eq:ac-def}
    \ac(K) \defeq  -\frac{\log(2)}{\log p_K}\,.
\end{equation}
Conversely, for fixed $\alpha$, we define $\kc(\alpha)$ as the unique solution to
\[
    \ac(K_c) = \alpha\,.
\]
We end this section by showing that for $\alpha$ even slightly greater than $\alpha_c$---or equivalently, $K$ slightly smaller than $K_c$---the expected size of $Z_{K,\alpha}$ is exponentially smaller than one. Applying Markov's inequality yields an upper-bound on the capacity of the symmetric perceptron, completing the elementary half of \cref{thm:main}.

\begin{proposition}\label{prop:Z-expansion}
Consider an arbitrary fixed $\alpha > 0$ and let $\kc := \kc(\alpha)$. There exist positive constants $\eps$, $c$, and $C$ depending only on $\alpha$ so that for all $K > 0$ with $|K - \kc|< \eps$, we have:
\begin{equation} \label{eq:Z-expansion-K}
    c \le \frac{\frac{1}{n}\log\E{|Z_{K,\alpha}|}}{K - \kc} \le C\,.
\end{equation}  
\end{proposition}

\begin{proof}
    Let $\ac := \ac(K)$. Recall that $\E{|Z_{K,\alpha_c}|} = 1$ by construction of $\alpha_c$, so that 
    \begin{equation}\label{eq:Z-alpha}
        \E{|Z_{K,\alpha}|} = \pa{p_K} ^{n(\alpha-\ac)}\,, \quad \text{i.e.} \quad 
        \frac{1}{n}\log\E{|Z_{K,\alpha}|} = (\log p_K)(\alpha-\ac)\,.
    \end{equation}
    For fixed $\alpha > 0$, $\kc$ is bounded away from 0. For $\eps$ small enough, $K$ is bounded from 0 also. It follows that $p_K$ and $p_{\kc}$ are bounded from zero as well, so that $\log p_K =: c_K$ and the first two derivatives in $K$ of $\log p_K$ are bounded in absolute value by some constants depending only on $\alpha$. (We have eliminated dependence on $\eps$ by recall that $\eps$ is a constant depending only on $\alpha$). Thus, we have by Taylor's expansion:
\begin{equation}\label{eq:equiv-perturb}
        \alpha - \alpha_c = \log(2) \pa{\frac{1}{\log p_K} - \frac{1}{\log p_{\kc(\alpha)}} } = c \pa{\kc - K} +   \cO{|\kc - K|^2}\,.
    \end{equation}
     Combining with \cref{eq:Z-alpha} yields the result. 
\end{proof}

We will repeatedly use \cref{eq:equiv-perturb} to justify the idea that studying $\alpha$ close to $\alpha_c$ and $K$ close to $\kc$ are equivalent. It will be convenient extract a formal statement for later use. 

\begin{lemma}[Equivalence of perturbations to $\alpha$ and $K$]\label{lemma:alpha-K}
    Let $\alpha := \alpha(n)$ and $K := K(n)$ be two sequences strictly bounded away from zero. There exist positive constants $\eps$, $c$, and $C$ so that if 
    \[ \limsup_{n \to \infty}\max\pa{|K-\kc(\alpha)|,~|\alpha-\ac(K)|} < \eps\,,
    \]
    then for all $n$ sufficiently large:
    \[
        c \le \frac{\ac(K)-\alpha}{K - \kc(\alpha)} \le C\,.
    \]
    \end{lemma}

The following are immediate consequences of applying Markov's inequality to \cref{prop:Z-expansion} and then invoking \cref{eq:Z-alpha}.

\begin{lemma}[First moment method] \label{lemma:markov} Fix $K > 0$ and let $\ac := \ac(K)$. 
For any $y > 0$,
\begin{equation} \label{eq:capacity-upper-tail}
    \PP{\cp n \ge \ac n + y} \le (p_K)^y\,.
\end{equation}
Alternatively, fix $\alpha > 0$ and let $\kc := \kc(\alpha)$. There exist positive constants $c$ and $\eps$ so that for all $y \in [0,\eps]$,
\begin{equation}\label{eq:disc-lower-tail}
    \PP{(\disc(A) - \kc)_{-} > y} < e^{-cny}\,.
\end{equation}
\end{lemma}

Integrating the tail bound \cref{eq:disc-lower-tail} easily yields a one-sided variance bound:
\begin{corollary} \label{prop:lower-var}
     \begin{equation*}
        \E{(\disc(A) - \kc)_{-}^2}^{1/2} = \cO{n^{-1}}\,.
    \end{equation*}
\end{corollary}

\begin{proof}[Proof of \cref{prop:lower-var}]
Integrating the tail bound \cref{eq:disc-lower-tail} for $y \in [0,\eps]$ and noting that trivially $(\disc(A) - \kc)_{-} \le \kc$ deterministically, 
\begin{align*}
    \E{(\disc(A) - \kc)_{-}^2} &\le \int_{0}^{\eps} 2y e^{-cny} dy + (\kc)^2 e^{-cn\eps} \le \cO{n^{-2}}\,.
\end{align*}
\end{proof}
In summary, the first moment method applied to $|Z_{K,\alpha}|$ shows $\disc(A)$ cannot be much smaller than $\kc$. Far more difficult is showing $\disc(A)$ is not much larger than $\kc$.

\subsection{Second Moment} We turn to the second moment method for an upper bound on $\disc(A)$. Fix $\alpha$ and let $Z_{K} := Z_{K,\alpha}$. If the variance of $|Z_{K}|$ is small, then $\disc(A)$ will concentrate around the first value $K$ for which $\E{|Z_K|} \ge 1$. 

Aubin, Perkins, and Zdeborov\'a \cite{APZ} established (originally contingent on a numerical hypothesis, later removed by \cite{ALS1})
\begin{theorem}[Coarse threshold; \cite{APZ}]\label{thm:2mm-APZ}
    Fix any $\alpha > 0$ and let $\kc := \kc(\alpha)$, $Z_K := Z_{K,\alpha}$. There exist positive constants $1 < C_1 < C_2 < \infty$ so that for any $\eps > 0$, for all $n$ sufficiently large,
    \[
        C_1 \le   \frac{\E{|Z_{\kc + \eps}|^2}}{\E{|Z_{\kc + \eps}|}^2}  \le C_2\,.
    \]
\end{theorem}
 An upper bound on $\disc(A)$ immediately follows. By the Paley-Zygmund inequality, for any constants $\alpha > 0$ and $\eps > 0$,
\[
    \liminf_{n \to \infty} \PP{\disc(A) < \kc + \eps} > 0\,.
\]
However, since $C_1 > 1$, \cref{thm:2mm-APZ} also establishes that the variance of $|Z_K|$ is too large to directly show that $\disc(A) < \kc + \eps$ with high probability. Of course, $|Z_K|$ having large variance certainly does not imply that $\disc(A)$ has large variance. A sharp threshold for $\disc(A)$ only requires that the random variable $\ind\pa{|Z_K|>0}$ is well-concentrated. 

Moving beyond the failure of the second moment method to establish a sharp threshold is often quite difficult, and indeed the results of Abbe, Li, and Sly \cite{ALS1}, and Perkins and Xu \cite{PX} required significant new ideas. Our work builds on \cite{PX}, which in turn builds on an idea of Talagrand \cite{talagrand1999intersecting}. Let us begin by formally stating the result of \cite{PX} that we wish to refine:

\begin{theorem}[\cite{PX}, Theorem 9]\label{thm:px}
    Fix $K > 0$ and $\eps > 0$. Let $Z_{\alpha} := Z_{K,\alpha}$ and $\ac:=\ac(K)$. For $\alpha \le \ac - \eps$ and for every $\delta > 0$, there exists $M = M(\delta, \eps, K)$ so that
    \begin{equation}\label{eq:px-goal}
        \limsup_{n \to \infty} \PP{\ba{\log \pa{ \frac{\ba{Z_{\alpha}}}{\E{\ba{Z_{\alpha}}}} } } \ge M \log n} \le \delta\,.
    \end{equation}  
\end{theorem}
Similarly to \cref{thm:2mm-APZ}, the original statement of \cref{thm:px} was originally contingent on a numerical conjecture (\cref{eq:F-conj} below) that can be removed by the results of \cite{ALS1}. Recall
\[
    \E{|Z_{\alpha}|} = \pa{p_K}^{n\pa{\alpha_c - \alpha}}\,,
\]
so that $\E{|Z_\alpha|} = \cE{\cOm{\eps n}}$ is exponentially large. In contrast, \cref{thm:px} yields that the fluctuations of $\log|Z_{\alpha}|$ are at most polynomial, establishing that $|Z_\alpha|$ is non-empty with high probability. The shortcoming we seek to improve is that \cref{thm:px} only holds for $|\alpha - \alpha_c|$ of constant order. In order to capture the critical window of the SBP, we need to show the bound of \cref{thm:px} actually continues to hold even as $|\alpha - \ac|$ approaches order $1/n$. We begin with the important first step of showing that this improvement holds with positive probability. \\

\begin{theorem}[Second moment method]\label{thm:2mm}
    Fix positive $K > 0$. There exist positive constants $1 < C_1 < C_2 < \infty$, so that for $n$ sufficiently large and any $\alpha \defeq  \alpha(n) \le \ac(K)$,
    \begin{equation}\label{eq:2mm-goal1}
        C_1 < \frac{\E{|Z_{K,\alpha}|^2}}{\E{|Z_{K,\alpha}|}^2} < C_2\,.
    \end{equation}
    Consequently,  for $A$ an $(\alpha,n)$-Gaussian matrix (see \cref{def:ensemble}),
    \begin{equation}\label{eq:2mm-goal2}
        \PP{\disc(A) \le \kc(\alpha)} > C_2^{-1} \,.
    \end{equation}    
\end{theorem}
\cref{eq:2mm-goal1} will be used as an important step in our strengthening of \cref{thm:px}, and \cref{eq:2mm-goal2} will be the main ingredient for our lower bound on the variance of $\disc(A)$. We prove \cref{thm:2mm} in the next section. 

In order to compute the second moment of $|Z_{K,\alpha}|$, we study a quantity we will call the \textit{second moment profile function} (sometimes called in physics the ``annealed free energy of pairs''). The second moment profile function is a measure of how much pairs of corners $(\sigma, \tau)$ from the discrete hypercube contribute to the second moment $\E{\ba{Z_{K,\alpha}}^2}$, as a function of the normalized Hamming distance $n^{-1}\sum_{i=1}^n \ind(\sigma_i \neq \tau_i)$. 

\begin{definition}[Pair probability]
    For $X$ and $Y$ independent standard normals, let
    \[
        q_K(\beta) \defeq  \PP{\ba{\sqrt{\beta}X + \sqrt{1-\beta}Y} \le K,~ \ba{ \sqrt{\beta}X - \sqrt{1-\beta}Y}  \le K}\,.
    \]
\end{definition}
Then $q_K(\beta)^{\alpha n}$ is the probability that $\sigma$ and $\tau$ with $\innerprod{\sigma}{\tau} = (2\beta-1)n$ are both in $Z_{K,\alpha}$.
\begin{definition}[Second moment profile function]
    Let $H$ be binary entropy; define the second moment profile function $F:[0,1] \to \RR$ by
    \[
        F(\beta) \defeq  F_{K,\alpha}(\beta) = H(\beta) + \alpha \log q_K(\beta)\,.
    \]
\end{definition}

\begin{figure}
    \centering
    \includegraphics[width=\textwidth / 2]{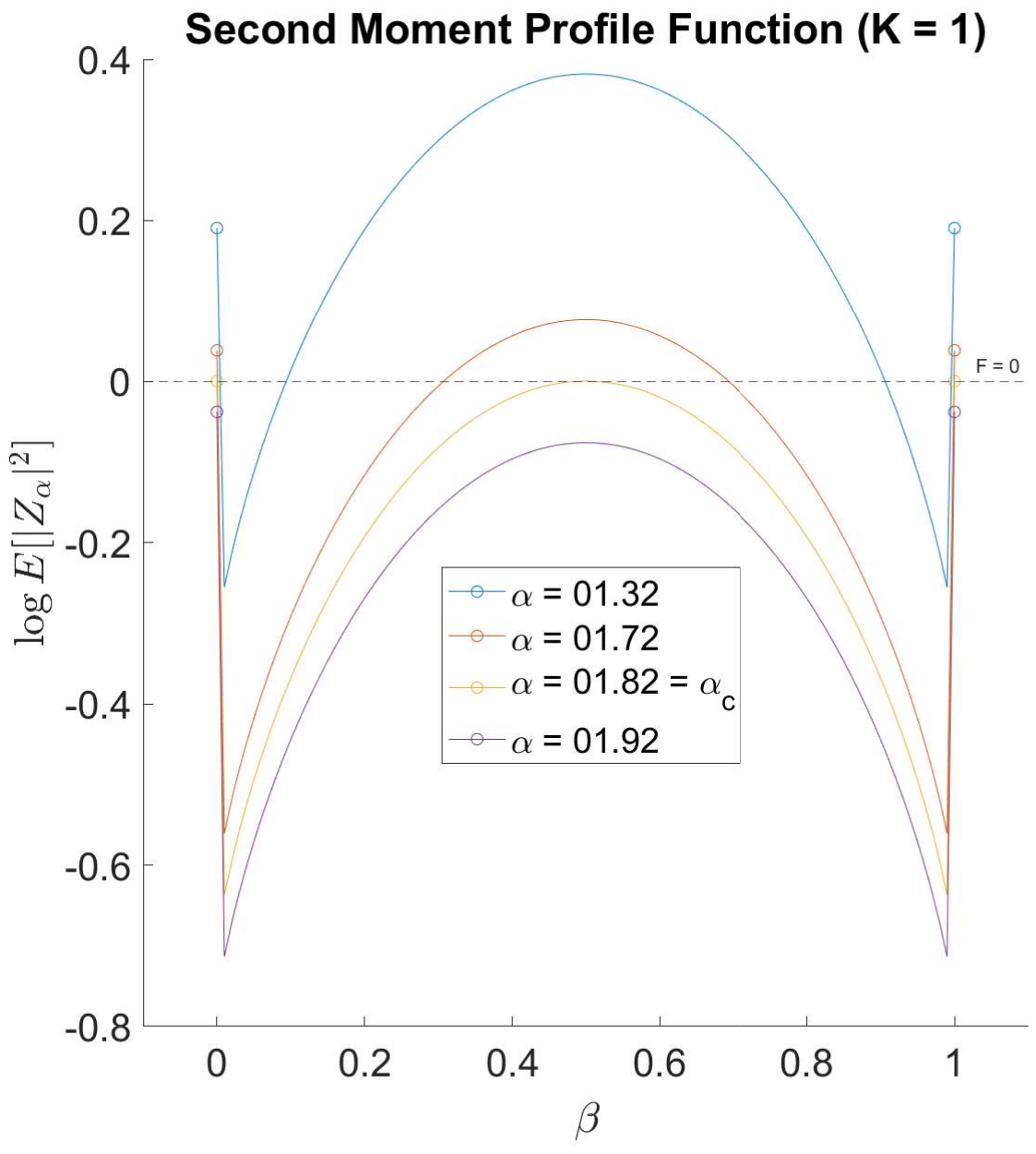}
    \caption{The shape of the second moment profile function $F(\beta)$ for $K = 1$ and various $\alpha$. (An identical plot has previously appeared in Figure 1a of \cite{APZ}). Note that $F(\beta)$ is symmetric around $\beta = 1/2$; $F(1/2) = 2F(0) = 2F(1)$; and most importantly, $F$ has a unique maximum at $\beta = 1/2$ for $\alpha \le \alpha_c$.}
    \label{fig:2mm}
\end{figure}

Let us collect some trivial facts about $F$ for intuition. First, Gaussian random variables which are orthogonal are also independent, so $q(1/2) = p^2$ for any $K$ and $\alpha$. Second, the functions $q(\beta)$ and $H(\beta)$ are symmetric around $\beta = 1/2$, so $F$ is as well. Finally, by construction of $\ac$, for any $K$ and $n$,
\[
    F_{K, \ac(K), n}\pa{\frac{1}{2}} = F_{K, \ac(K)}\pa{0} = - \log(2)\,.
\]
Regarding the shape of $F$, the following was conjectured and assumed in \cite{APZ} and \cite{PX} (as well as an earlier version of \cite{ALS1}) in order for second moment method computations to be tractable:
\begin{conjecture}[Shape of $F$; conjecture]
    For all positive $K$ and $\alpha$ with $F_{K,\alpha}''(1/2) < 0$, it further holds that $F \defeq  F_{K,\alpha}$ has a single critical point in $(0,1/2)$. In particular, for any $0 \le a \le b \le 1/2$,
    \begin{equation}\label{eq:F-conj}
        \max_{a \le \beta \le b} F(\beta) \in \max\cb{F(a),~F(b)}\,.
    \end{equation}
\end{conjecture}
Compelling numerical evidence was supplied in \cite{APZ}. Later, in \cite{ALS1}, a slightly weaker form of this conjecture was recovered, for $K >0$ and $\alpha < \ac$. The main difference is that for some $K$, it is only easy to rigorously establish that $F$ is decreasing near $0$ and increasing near $1/2$. In between, it is far easier to simply show that $F$ is much less than $F(1/2)$---rather than controlling the number of critical points---which is more than enough for second moment method applications. We now collect the facts, rigorously established in \cite{ALS1}, that we will need: 

\begin{lemma}[Shape of $F$ \cite{ALS1}]\label{lemma:free-energy}
Recall $F \defeq  F_{K,\alpha}$.
\begin{enumerate}
    \item For any $\alpha$ and $K$
    \begin{equation}\label{eq:F''(1/2)}
        F''\pa{\frac{1}{2}} =  4\pa{-1 + \frac{2}{\pi}\frac{\alpha K^2 e^{-K^2}}{p^2}}\,.
    \end{equation}
    
    \item Fix any $K > 0$. There exists $\eps \defeq  \eps(K) > 0$ sufficiently small so that for any $\alpha \le \ac(K)$, 
    \begin{equation}\label{eq:F''-ub}
        F''\pa{\frac{1}{2}} < -\eps\,.
    \end{equation}
    
    \item There exists $b\defeq  b(K) > 0$ such that for all $\alpha \le \ac$, $F_\alpha(\beta)$ is decreasing for $\beta \in [0,b]$. 
    
    \item Let $K > 0$ and $F_{\ac} \defeq  F_{K,\ac(K)}$. There exists $\eps \defeq  \eps(K) > 0$ such that, for any $x,y \in [0,1/2]$,
    \begin{equation}\label{eq:F-max}
        \max_{\beta \in [x, y]} F_{\ac}(\beta) \le \max \{F_{\ac}(x),~F_{\ac}(y),~ F_{\ac}(1/2)-\eps\}\,.
    \end{equation}
    
\end{enumerate}
\end{lemma}
The proof of \cref{lemma:free-energy} is deferred to the appendix, since it is directly reproduced from \cite{ALS1} with only trivial modifications. Upon first read, one could also simply assume \cref{eq:F-conj}. \\

\subsection{Sharp Threshold}
We seek to boost the ``positive probability'' guarantee of \cref{thm:2mm} into a ``high probability'' guarantee. Our goal is: if $\log\pa{ \E{\ba{Z_{K,\alpha}}} } \gg \log n$, then $|Z_r| \gg 1$ with high probability. Making this precise and also giving a quantitative tail bound, we will show:
\begin{theorem}[Upper tail]\label{thm:sharp-transition}
    Fix $K > 0$. There exists a constant $c >0$ so that if 
    \[
        \alpha \defeq  \alpha(n) = \ac(K) - xn^{-1}\log(n)
    \]
    for a sufficiently large constant $x > 0$, then for all $n$ sufficiently large:
    \begin{equation}\label{eq:st-goal1}
        \PP{|Z_{K,\alpha}| = 0} < n^{-cx}\,.
    \end{equation}
    Equivalently, for any $\alpha > 0$ there exists some constant $c > 0$ so that for any fixed $x$ sufficiently large, for all $n$ sufficiently large:
    \begin{equation}\label{eq:st-goal2}
        \PP{\disc(A) > K_c + \frac{x\log(n)}{n}} < n^{-cx}\,.
    \end{equation}
\end{theorem}
It is fairly straightforward to control the variance of $\disc(A)$ by integrating this tail bound. The only (small) complication is that our tail bound does not allow for $x$ growing as a function of $n$. 
\begin{corollary}\label{cor:var+}
    Fix $\alpha > 0$ and let $n$ be sufficiently large. We have the one-sided variance bound
    \begin{equation}\label{eq:var+}
        \E{(\disc(A) - \kc)_{+}^2}^{1/2} = \cO{\frac{\log(n)}{n}}\,.
    \end{equation}
\end{corollary}

\begin{proof}[Proof of \cref{cor:var+}]
A crude bound suffices. We can always upper-bound $\disc(A)$ by $\|A\sigma\|_\infty$ for some arbitrary $\sigma$. Recall that $A\sigma$ is a centered Gaussian vector with an $\alpha n \times \alpha n$ identity covariance matrix. It is easy to check by a direct computation the very sub-optimal bound: $\E{\|A\sigma\|_\infty^4} < n^4$. Define
\[
    U\defeq  (\disc(A) - \kc)_{+}\,.
\]
Then for sufficiently large $x>0$, we have by Cauchy-Schwarz:
\begin{align*}
    \E{U^2}^{1/2} &= \E{U^2 \ind(U \le xn^{-1}\log n)}^{1/2} + \E{U^2 \ind(xn^{-1}\log n \le U)}^{1/2} \\
    &\le  \frac{x\log(n)}{n} + \E{\|A\sigma\|_\infty^4}^{1/4} \PP{\disc(A) > \kc + xn^{-1}\log n}^{1/4} \\
    &\le \frac{x\log(n)}{n} + n\pa{n^{- cx/4}} \\
    &= \cO{\frac{\log(n)}{n}}\,.
\end{align*}

\end{proof}

\section{Second Moment Method}\label{section:UB-2}

Here we establish that solutions exist up to criticality with positive probability.
  
\begin{proof}[Proof of \cref{thm:2mm}] 
Fix $K > 0$ and some sequence $\alpha \defeq  \alpha(n) \le \ac$. Our goal is to upper bound the ratio
\begin{equation}\label{eq:2mm-ratio}
    \frac{\E{|Z_{K,\alpha}|^2}}{\E{|Z_{K,\alpha}|}^2} = 2^{-n}\sum_{s = 0}^n \binom{n}{s} \pa{ \frac{q\pa{\frac{s}{n}}}  {p^2} }^{\alpha n}\,.
\end{equation}
We claim that \cref{eq:2mm-ratio} is monotone increasing in $\alpha$, so that if we establish the claim for $\alpha$ set to equal $\ac$, then the claim also follows for all $\alpha \le \ac$. Indeed, for any $\beta$, it easy to check that $p^2 \le q(\beta)$ by a direct computation.  Alternatively, if the reader is familiar with such tools, the Gaussian Correlation Inequality (e.g. Theorem 1 of \cite{corr}) also immediately yields $p^2 \le q(\beta)$. Indeed, $q(\beta)$ is the probability that the Gaussian vector $A_i$ is in the intersection of two origin-symmetric convex slabs:
\[
    \cb{a: |\innerprod{a}{\sigma}| \le K} \cap \cb{a: |\innerprod{a}{\tau}| \le K}\,,
\]
where $\sigma$ and $\tau$ are corners of the discrete cube with Hamming distance $\beta n$. 

So, it suffices to upper bound \cref{eq:2mm-ratio} for $\alpha = \ac$. Note that for this choice of $\alpha$, the denominator of the left of \cref{eq:2mm-ratio} is simply one. For convenience, we adopt the shorthand $F \defeq  F_{K,\ac(K), n}$. \\

Let $m \defeq  \ac n$ be the number of rows in the matrix $A$. Consider $\beta_s \defeq  \frac{1}{2} + \frac{s}{2n}$. We would like to Taylor expand $q(\beta_s)$; first note that $q(0) = p^2$. Next, 
\begin{align*}
    q'(\beta) &= \frac{1}{2\pi} \partial_\beta \int_{x = -K}^K \int_{y = \frac{-K - (2\beta-1) x}{2\sqrt{\beta(1-\beta)}}} ^ {\frac{K - (2\beta-1) x}{ 2\sqrt{\beta(1-\beta)} }} \cE{-\frac{x^2+y^2}{2}}dydx \\
    &= \frac{1}{\pi  \sqrt{\beta(1-\beta)}} e^{-\frac{K^2}{2 (1-\beta)}} \left(e^{\frac{(2 \beta - 1) K^2}{2 \beta (1 - \beta) }}-1\right) \,.
\end{align*}
Note that $q'(1/2) = 0$. Differentiating again and then evaluating at $\beta = 1/2$ yields:
\begin{align*}
    q''(1/2) &=  \frac{8 e^{-K^2} K^2}{\pi } = 4p^2\pa{1 - \mu_2}^2\,,
\end{align*}
where we have adopted the following notation of \cite{ALS1}:
\begin{equation}\label{eq:mu2}
    \mu_2 \defeq  p^{-1}\int_{-K}^K x^2 e^{-x^2/2}\frac{dx}{\sqrt{2\pi}} = -p^{-1}\sqrt{\frac{2 }{\pi}}Ke^{-K^2/2} + 1\,.
\end{equation}
Thus, restricted to $|\beta_s| < \eps$ for some small constant $\eps > 0$ depending only on $K$, Taylor's expansion yields
\begin{equation}\label{eq:q-taylor}
    q(\beta_s) = p^2 \pa{1 + \frac{(1- \mu_2)^2(s^2/n)}{2n} + \cO{\frac{s^4}{n^4}} }\,.
\end{equation}
We are ready to bound \cref{eq:2mm-ratio}. Split the sum into three annular regions and treat them separately. (For the first region, highly similar computations appear in e.g. Lemma 3.1 in \cite{ALS1}; Lemma 7 of \cite{APZ}; or Lemma 3 of \cite{dja-jnw}.)
\begin{enumerate}
    \item ($1/2 - \del \le \beta \le 1/2$) Recall $\beta_s \defeq  1/2 + s/(2n)$. Let $\del > 0$ be some arbitrarily small constant that we fix later. For $\del$ sufficiently small, we apply Stirling's formula and \cref{eq:q-taylor}. Denote for shorthand $B \defeq  \sqrt{\alpha}(1-\mu_2)/2$. Counting over even-valued $s$, 
    \begin{align*}
        I_1 &\defeq  \sum_{|s| \le 2\del n } 2^{-n} \binom{n}{\frac{n+s}{2}} \pa{\frac{q(\beta_s)}{p^2}}^m \\
        &= \sqrt{\frac{2}{n\pi}}\sum_{|s| \le 2\del n } \cE{-\frac{s^2}{2n}\mO{\del}} \cE{\frac{\alpha}{2}(1-\mu_2)^2\pa{\frac{s^2}{n}} \mO{\del} }\\
        &= \sqrt{\frac{2}{n\pi}}\sum_{|s| \le 2\del n } \cE{-\frac{s^2}{2n}\pa{1 - 4B^2}\mO{\del}}e^{-2B^2}\,.
    \end{align*}
    Since $\delta$ can be taken arbitrarily small, it suffices to show $|B|$ is strictly less than $1/2$, uniformly for sufficiently large $n$. This follows immediately by combining \cref{eq:F''(1/2)} with \cref{eq:mu2}. Indeed, by \cref{eq:F''-ub}, there is some $\eps \defeq  \eps(K) > 0$ so that
    \[
        1 - 4B^2 = -\frac{1}{4}F''\pa{\frac{1}{2}} > \eps > 0 \,.
    \]
    Thus $|B|$ is strictly bounded from $1/2$, yielding 
    \begin{equation}\label{eq:2mm-exact}
        I_1 = \mO{\del} e^{-2B^2} \sqrt{\frac{1}{1-4B^2 } }   \,. 
    \end{equation}
    (The factor of $2$ from $\sqrt{2/(n\pi)}$ was cancelled by the fact that we were only summing over even $s$). Taking $\del$ sufficiently small, we obtain for some implicit constant depending only on $K$ that $I_1 = \cO{1}$. \\

    \item ($\delta \le \beta \le 1/2 - \del$)
    Let $I_2$ be 
    \begin{align*}
        I_2 &\defeq  \sum_{s:~2\delta n \le |s| < (1-2\delta)n } 2^{-n} \binom{n}{\frac{n+s}{2}} \pa{\frac{q(\beta_s)}{p^2}}^m \,.
    \end{align*}
    Applying Stirling's formula, 
    \begin{align*}
        I_2 &< n(2^{-n}p^{-2m}) \max_{\beta \in [\del,\frac{1}{2}-\del]} \cE{nF(\beta) + \cO{\log n}}  \\
        &= \max_{\beta \in [\del,\frac{1}{2}-\del]} \cE{n\pa{F(\beta) - F(1/2)} + \cO{\log n}} \,.
    \end{align*}
    It suffices to show the exponent is strictly negative and of order $n$ on this interval. Applying \cref{eq:F-max}, we have for some $\eps_0 \defeq  \eps_0(K,\delta) > 0$,
    \begin{align*}
        \max_{\beta \in [\del, 1/2-\del]} F(\beta) - F(1/2) &\le \max\cb{F(\del) - F(1/2),~F(1/2-\del)- F(1/2), -\eps} \le -\eps_0 \,.
    \end{align*}
    Thus for some implicit constant depending only on $K$ and $\delta$, $I_2 = \cE{-|\cT{n}|}$. \\
    
    \item ($0 \le \beta < \del$)
    Define $I_3$, the contribution of the remaining terms, by
    \begin{align*}
        I_3 &\defeq  \sum_{|s| > (1- 2\del) n} 2^{-n} \binom{n}{\frac{n+s}{2}} \pa{\frac{q(\beta_s)}{p^2}}^m \,.
    \end{align*}
    By Stirling's formula---and recalling we set $\alpha$ to be $\ac(K)$, where $\ac$ is given in \cref{eq:ac-def}---we also have the identity
    \begin{equation}\label{eq:2mm-outer-stirling}
        I_3 = \sum_{|s| > (1- 2\del) n} p^{-m}\cE{nF(\beta_s) + \cO{\log n}}   \,.
    \end{equation}
    
    Since we are able to take $\delta$ an arbitrarily small constant in all the previous parts of this proof, let $\delta$ be sufficiently small so that by \cref{lemma:free-energy}, $F$ is decreasing on $[0,\del]$. The idea is to say $F(\beta)$ is already extremely negative even for $\beta = 1/n$. Intuitively, such a dramatic decay is possible due to the ``frozen'' nature of typical solutions or, equivalently, the fact that $F'(\beta) \to -\infty$ as $\beta \to 0$ from above. Note the minor subtlety that $F$ being decreasing does not directly imply that the summand of $I_3$ is decreasing, due to the logarithmic error terms in Stirling's formula. Conveniently, $F(1/n)$ is negative enough to outweigh these lower-order corrections. 
    
    First computing the summand for $\beta = 0$, since $q(0) = p$, we have:
    \begin{align*}
        2^{-n} \binom{n}{0} \pa{\frac{q(0)}{p^2}}^m = 2^{-n} p^{-m} = 1 \,.
    \end{align*}
    Next, in order to compute the summand for $\beta = 1/n$, we claim: (c.f. Lemma 6 of \cite{PX} for a related result on the ``planted'' model)
    \begin{equation}\label{eq:q-decrease-0}
        \frac{ q\pa{1/n} }{ q\pa{0} } =  \frac{ q\pa{1/n} }{ p } \le 1 - \cOm{\frac{1}{\sqrt{n}} } \,.
    \end{equation}
    Indeed, let $\beta \defeq  1/n$ and define $\Delta \defeq  c_0n^{-1/2}$ for some new positive constant $c_0$ that we fix to be sufficiently small later. We have by definition of $q$:
    \begin{align*}
        q(\beta) &= \frac{1}{2\pi}\int_{-K}^{K} e^{-y^2/2} \int_{\frac{-K + (1-2\beta)y}{2\sqrt{\beta(1-\beta)}}}^{\frac{K + (1-2\beta)y}{2\sqrt{\beta(1-\beta)}}} \cE{-x^2/2}dxdy \\
        &\le \frac{1}{\sqrt{2\pi}}\int_{-K}^{K-\Delta} e^{-y^2/2} dy + \frac{1}{2\pi}\int_{K- \Delta}^{K} e^{-y^2/2} \int_{\frac{-K + (1-2\beta)y}{2\sqrt{\beta(1-\beta)}}}^{\frac{K + (1-2\beta)y}{2\sqrt{\beta(1-\beta)}}} \cE{-x^2/2}dxdy \\
        &\le p_K - \Delta \pa{ \frac{e^{-K^2/2}}{\sqrt{2\pi}} } + \frac{1}{2\pi} \int_{K-\Delta}^K e^{-y^2/2} \int_{\frac{-K + (1-2\beta)(K-\Delta)}{2\sqrt{\beta(1-\beta)}}}^{\infty} \cE{-x^2/2}dxdy \,.
    \end{align*}
    Examining the bounds of integration,
    \begin{align*}
        \frac{-K + (1-2\beta)(K-\Delta)}{2\sqrt{\beta(1-\beta)}} &= \frac{-2K\beta - \Delta}{ 2\sqrt{1/n} } \mO{\frac{1}{n}} = -\frac{\Delta \sqrt{n}}{2}\mO{\frac{1}{\sqrt{n}}} \,.
    \end{align*}
    Thus, we obtain in total:
    \begin{align*}
        q\pa{\frac{1}{n}} &\le p_K - \Delta \pa{ \frac{e^{-K^2/2}}{\sqrt{2\pi}} } + \frac{1 + \cO{n^{-1}} }{\sqrt{2\pi}} \int_{K-\Delta}^{K} e^{-y^2/2} \pa{ \frac{1}{2} + \int_{- \frac{c_0}{2} }^{0} \cE{-x^2/2} dx}dy  \\
        &=  p_K  + \Delta  \pa{- \frac{e^{-K^2/2}}{\sqrt{2\pi}} + \frac{e^{-K^2/2} } {2\sqrt{2\pi}} + \cT{c_0e^{-K^2/2}} } + \cO{\frac{1}{n}} \\
        &\le p_K - \cT{\frac{1}{\sqrt{n}}} \,.
    \end{align*}
    The last line follows by taking $c_0$ sufficiently small. We have thus established the desired bound \cref{eq:q-decrease-0}. Raising this inequality to the $m$'th power yields 
    \[
       q(1/n)^m \le p^{m}\cE{-\cOm{\sqrt{n}}} \,,
    \]
    which is far smaller than even what is needed. Indeed, returning to \cref{eq:2mm-outer-stirling} and recalling that $F$ is strictly decreasing, as well as noting crudely that $H(1/n) < n^{-.9}$,  we obtain:
    \begin{align*}
        I_3 &= \sum_{s = 0}^{\del n} p^{-m}\cE{nF(\beta_s) + \cO{\log n}}  \\
        &\le 2\pa{ 1 + \del n  \pa{\frac{q(1/n)}{p}}^{m} \cE{n^{.1} + \cO{\log n}} }\\
        &\le 2\pa{1 + \cE{-\cOm{\sqrt{n}}}} \,.
    \end{align*}
    Thus, $I_3 = 2 + \oo(1)$. 
\end{enumerate}
This completes our casework. Summing $I_1$, $I_2$, and $I_3$, we have shown that the second moment ratio \eqref{eq:2mm-ratio} is $\cO{1}$. By our previous considerations on the monotonicity of \cref{eq:2mm-ratio} with respect to $\alpha$, \cref{eq:2mm-goal1} is then established for all $\alpha\defeq \alpha(n) \le \ac$. By the Paley-Zygmund inequality, \cref{eq:2mm-goal2} follows immediately and the lemma is complete. 
\end{proof}

\section{Sharp Threshold}\label{section:UB-ST}

\subsection{Sketch of \cref{thm:sharp-transition} }\label{sketch}
By \cref{lemma:alpha-K} (``equivalence of perturbations to $\alpha_c$ and $K_c$''), \cref{eq:st-goal1} implies \cref{eq:st-goal2}. In order to establish  \cref{eq:st-goal1}, we need to refine \cref{thm:px} of Perkins and Xu to hold even for $\alpha$ allowed to vary with $n$. In particular, it should hold for $\alpha \le \alpha_0 \defeq  \ac - x n^{-1}\log(n)$. We also would like quantitative tail-bounds. Let us briefly survey the proof.

We follow an approach that Talagrand developed for the binary perceptron \cite{talagrand1999intersecting}. Talagrand's approach was carefully adapted to the setting of the symmetric perceptron by Perkins and Xu \cite{PX}. The outline is to build a constraint satisfaction problem by adding one clause at a time, and then show the logarithm of the number of solutions concentrates via martingale arguments. 

Reusing the notation of \cite{PX} for clarity of comparison, define a time-indexed process $(S_t)_{t}$ of the solutions to the first $t$ rows:
\begin{equation}\label{eq:s-def}
    S_t \defeq  \{\sigma \in \{\pm 1\}^n: |(A\sigma)_j| \le K,~ \forall~1 \le j \le t\}, \quad \E{|S_t|} = 2^n (p_K)^t\,.
\end{equation}
In terms of our previous notation, $S_{t}$ is equal to $Z_{K,\alpha}$ with $\alpha = t/n$. Denote the deviation of $\log|S_t|$ from its mean by $Q_t$, which we write as a telescoping sum:
\[
    Q_t \defeq  \log\pa{\frac{|S_t|}{\E{|S_t|}}} = \sum_{i=1}^t \left[ \log\pa{\frac{|S_i|}{\E{|S_i|}}} -  \log\pa{\frac{|S_{i-1}|}{\E{|S_{i-1}|}}} \right] \,.
\]
Rewriting $Q_t$ in terms of a martingale that we call $Y_t$, 
\[
    Q_t = \sum_{i=1}^t \log(1+Y_i), \quad Y_t \defeq  \frac{1}{p}\pa{\frac{|S_t|}{|S_{t-1}|} - p} \,.
\]
Then, as long as $Y_i$ is small for each $i \le \alpha_0 n$, we can use the Taylor Expansion
\[
    Q_t = \sum_{i=1}^t Y_i - \frac{Y_i^2}{2} + \cO{Y_i^3} \,.
\]
The $Y_i$ are centered and should be roughly $Y_i \asymp n^{-1/2}$, so that both the sum of the $Y_i$ and the sum of the $Y_i^2$ are constant order, and the sum of the $Y_i^3$ is vanishing. We then expect $|Q_t| = \cO{\log(n)}$ with exponential tails, which would yield the theorem. Indeed, $\log\E{\ba{S_{\alpha_0n}}}$ is at least $\cOm{x \log(n)}$, so if we take $x$ sufficiently large, then $|Q_t| \ll \log\E{\ba{S_{t}}}$ for all $t \le \alpha_0 n$ with very high probability. This would imply that the fluctuations of $|S_t|$ are much less than its mean, yielding the result. More precisely, we would have with high probability:
\begin{align*}
    \log|S_{\alpha_0n}| &=  \log\E{|S_{\alpha_0n}|} + \log \pa{\frac{|S_{\alpha_0n}|}{\E{|S_{\alpha_0n}|}}}\\
    &= \log \E{|S_{\alpha_0n}|} + Q_{\alpha_0 n} \\
    &\ge c x \log n - C \log n \\
    &\gg 0\,.
\end{align*}
The last line is achieved by taking $x$ sufficiently large, completing the theorem. \\

The analysis is by induction. The key idea for this induction is a geometric notion of regularity for $S_t$ that asserts solutions are sufficiently ``well-spread''. We begin by assuming for induction that $S_{t'}$ is \textit{regular} for all $t' \le t$ (precise definition given shortly). Then we follow with three estimates (c.f. Lemmas 11, 12, and 13 of \cite{PX} respectively). First, if $S_t$ is regular, then $Y_{t+1}$ enjoys good tail bounds. Second, if $Y_{t'}$ has good tails for all $t' \le t+1$, then martingale concentration yields exponential tail bounds for $Q_{t+1}$. Third, if $Q_{t+1}$ is small, then $S_{t+1}$ is also regular with good probability. 

We emphasize that this structure was innovated in \cite{talagrand1999intersecting} and refined in \cite{PX}. We will directly use the first estimate (tail bound on $Y_t$) from \cite{PX}. For the second and third estimates, we also significantly borrow from the structure of the proofs in \cite{PX}, albeit with exponential improvements. This concludes our sketch. \\

\subsection{Proof of \cref{thm:sharp-transition} (c.f. Theorem 9 of \cite{PX})} 

Let $S_t$, $Q_t$, and $Y_t$ be as defined in the sketch. Fix $K> 0$. All other constants in what follows will depend on $K$; we treat $K$ as a universal constant and suppress this dependence in our notation. Fix $x$ sufficiently large. Define $\alpha_0 \defeq  \alpha_0(n)$ by $\alpha_0n \defeq  \ac n - x\log(n)$. We emphasize that no constants in what follows will depend implicitly on $x$. Recall our convention that $c$ and $C$ denote generic constants that may vary between lines; important constants will have a subscript. \\

In order to use the martingale structure of $Y_t$, define the filtration $\mathcal{F}_t$ generated by revealing the first $t$ rows of the matrix $A$. Define the conditional measure $\PP[t]{\cdot}$ by
\begin{equation}
    \PP[t]{\cdot} := \PP{\cdot \,|\,\mathcal{F}_t}\,.
\end{equation}
With this notation, we are ready to formally state our main definitions and lemmas.

\begin{definition}[Regular] 
We say $S_t \subset \{\pm 1\}^n$ is \textbf{regular} if $S_t \neq \emptyset$ and, for $\sigma_t^{(1)}$, $\sigma_t^{(2)}$ drawn uniformly and independently (with replacement) from $S_t$,
\[
    \PP[t]{ \ba{ \innerprod{\sigma_t^{(1)}}{\sigma_t^{(2)}}} >   \sqrt{\creg n (x\log(n) + |Q_t|)}} \le n^{-c_rx} \,,  
\]
where $\creg$ and $c_r$ are some positive universal constants.
\end{definition}

The parameters $C_r$ and $c_r$ are explicitly chosen below in \cref{eq:C_r-def} and \cref{eq:c_r-def} respectively. Define two stopping times: $\tau_{S}$ the first time $t$ that $S_t$ is not regular and $\tau_Q$ the first time that $|Q_t|$ is large. Also define the first time $\tau$ that either of these ``bad'' events occurs. More precisely, for $C_q \in (0,1)$ some small universal constant (explicitly chosen in \cref{eq:C_q-def} below),
\begin{align*}
    \tau_S  &\defeq \min\cb{t:\,S_t \text{ is not regular}}\,, \\
    \tau_Q  &\defeq \min\cb{t:\,|Q_t| > C_q x \log(n)}\,, \\
    \tau &\defeq \tau_S \wedge \tau_Q \,.
\end{align*}
The first estimate of the induction is that if $S_t$ is regular then $Y_{t+1}$ is sub-exponential. Up to some trivial rewriting (we use $\tau_Q$ to suppress $|Q|$), the following is exactly Lemma 11 of \cite{PX}.
\begin{lemma}[$Y_{t+1}$ is small; \cite{PX}, Lemma 11]\label{lemma:Y-improved} There exist positive constants $c$ and $C$ so that for any $t$, for all sufficiently large $n$ and $y$,
\begin{equation}\label{eq:Y-tail}
    \ind\pa{\tau > t}\PP[t]{|Y_{t+1}| >  Cy \sqrt{\frac{(1 + C_qx)\log n}{n}}} \le \cE{- cy } \,.
\end{equation}
\end{lemma}
Next, we would like to say that $Q_t$ is small for $t \le \tau_S$. Only a first-moment bound on $|Q_t|$ is given in \cite{PX}; the refinement we require is a bound on the moment-generating function. This will both allow us take $t$ much closer to $\ac$ as well as yield strong tail bounds on $|S_t|$.

\begin{lemma}[$Q_{t+1}$ is small]\label{lemma:Q-improved} There exists a constant $c>0$ so that for all sufficiently large $n$ and any $t \le \alpha_0 n$,
\begin{equation}\label{eq:Q-tail-early}
    \PP{|Q_{ t \wedge \tau}| > C_qx \log n} \le n^{-cx} \,.
\end{equation}
\end{lemma}

Assuming $S_{t}$ is regular, these lemmas yield that $|Y_{t+1}|$ and $|Q_{t+1}|$ are small. The induction is complete if $|Q_{t+1}|$ being small implies $S_{t+1}$ is regular.

\begin{lemma}[$S_{t+1}$ is regular; $\tau$ is large]\label{lemma:tau-improved}
There exists a constant $c>0$ so that for all sufficiently large $n$,
\[
    \PP{\tau_S \le \alpha_0n} < n^{-cx}  \,.
\]
\end{lemma}

We prove these lemmas in the next subsection. Let us first formally check that they imply \cref{thm:sharp-transition}, our desired bound on the upper tail of $\disc(A)$. 
\begin{proof}[Proof of \cref{thm:sharp-transition}]
For convenience, define the constant $c_K := -\log(p_K) > 0$. Since $K$ is fixed, $c_K$ may be treated as a universal constant. By \cref{eq:Z-alpha}, 
\begin{equation}\label{eq:C_q-constraint}
    \log\E{|S_t|} = c_K(\alpha_c n - t) \ge c_K x\log n,\quad \forall\, t \le \alpha_0n \,.
\end{equation}
We will eventually choose $C_q$ explicitly (in \cref{eq:C_q-def}) so that $C_q \le c_K/2$. Thus, on the event $\cb{\tau_Q > t}$, we have $|S_t| \gg 0$ almost surely. Indeed, as sketched in \cref{sketch},
\begin{align*}
    \log|S_{t}| = Q_{t} + \log \E{|S_{t}|} \ge -C_q x \log n + c_K x \log n \ge \frac{c_K}{2}x \log n\,,
\end{align*}
where the equality is by definition of $Q_t$ and the first inequality holds almost surely on the event $\cb{\tau_Q > t}$.

The theorem will follow if the event $\cb{\tau_Q > \alpha_0 n}$ holds with very high probability. By \cref{lemma:tau-improved}, with probability at least $1 - n^{-cx}$, we have $\cb{\tau_S > \alpha_0 n}$, i.e. the solution space is regular simultaneously for all $t \le \alpha_0 n$. Then, by \cref{lemma:Q-improved},
\begin{align*}
    \PP{\tau_Q \le \alpha_0 n} &\le \PP{\tau_S \le \alpha_0 n} + \sum_{t = 1}^{\alpha_0 n}\PP{\tau_S > \alpha_0 n, \tau_Q = t } \\
    &\le  n^{-cx} + \sum_{t = 1}^{\alpha_0 n} \PP{|Q_{t \wedge \tau}| > C_q x \log n } \\ 
    &\le n^{-cx}\,.
\end{align*}
In summary, we have shown:
\[
    \PP{|Z_{K,\,\ac n - x\log n}| = 0} \le \PP{|S_{\alpha_0n}| \le e^{\frac{c_K}{2} x\log n}} \le  \PP{\tau_Q \le \alpha_0 n} \le n^{-cx}\,.
\]
\end{proof}

\subsection{Proofs of Lemmas}

\begin{proof}[Proof of \cref{lemma:Q-improved}]
Fix $t \le \alpha_0n$ and define the event $E$ that all the $Y_i$ are bounded for $i \le t \wedge \tau$: 
\[
    E := \bigcap_{i \le t-1}~\cb{ \ind_{\tau > i} |Y_{i+1}| < n^{-.4}}\,.
\] 

By \cref{lemma:Y-improved} and union bound, the event $E$ holds with probability at least $1 - n^{-C_E}$ for any (arbitrarily large) constant $C_E$. Let $\lambda \in (0,1)$, where $\lambda$ does not vary with $n$ and will be taken sufficiently small later; our goal is to show
\begin{equation}\label{eq:Q-pre-goal}
    \E{\cE{\lambda|Q_{ t \wedge \tau }|} \ind_E} \le n^{\lambda^2 Cx} \,.
\end{equation}
Let us simplify \cref{eq:Q-pre-goal}. Since $e^{Q_t}$ is a martingale for all $t$ by construction and $\lambda \in (0,1)$, the optional stopping theorem and Jensen's inequality yield
\[
    \E{\cE{\lambda Q_{t\wedge \tau }} } \le 1\,.
\]
Thus, to establish \cref{eq:Q-pre-goal}, it suffices to show:
\begin{equation}\label{eq:Q-mgf-goal}
    \E{\cE{-\lambda Q_{t \wedge \tau }} \ind_E} \le n^{\lambda^2 Cx}.
\end{equation}
By a Taylor expansion, we have on the event $E$ that
\begin{equation}\label{eq:Y-small-taylor}
    \cE{-\lambda \log(1+Y_i)} \le 1 - C\lambda Y_i + C\lambda^2 Y_i^2, \quad \forall\, i \le t \wedge \tau \,.
\end{equation}
Additionally, by integrating the tail bound of \cref{lemma:Y-improved}, we obtain for some new positive constant $C > 0$: 
\begin{equation}\label{eq:Y^2}
    \ind_{\tau > i}\E[i]{Y_{i+1}^2} \le C n^{-1} (1 + C_qx)\log n, \quad \forall\, i \,.
\end{equation}

Recall that $(Y_i)_i$ is a martingale under the filtration generated by revealing the rows of $A$. Combining \cref{eq:Y-small-taylor} and \cref{eq:Y^2} yields
\begin{align*}
    &\E{\cE{-\lambda Q_{t\wedge \tau}} \ind_E}  \\ &= \E{ \cE{- \lambda Q_{(t-1) \wedge \tau} - \ind_{\tau > t-1} \lambda \log(1 + Y_t)}  \ind_E } \\
    &= \E{ \cE{- \lambda Q_{(t-1) \wedge \tau}} \E[t-1]{ \cE{ -\ind_{\tau > t-1} \lambda \log(1 +  Y_t) } }  \ind_{E} }  \\
    &\le \E{ \cE{- \lambda Q_{(t-1) \wedge \tau}} \pa{1 -  C \ind_{\tau > t-1} \lambda \E[t-1]{Y_t } + C \ind_{\tau > t-1} \lambda^2 \E[t-1]{Y_t^2} } \ind_E }  \\
    &\le \E{\cE{- \lambda Q_{(t-1)\wedge \tau}} \ind_E \pa{1 + \lambda^2 \frac{C(1 + C_qx)\log n)}{n}}} \,.
\end{align*}
We have used that the linear term of $Y$ in \cref{eq:Y-small-taylor} vanishes in expectation. Since $Q_0 = 0$ by definition, we obtain the desired result \eqref{eq:Q-mgf-goal} by an easy induction:
\begin{equation*}
    \E{\cE{-\lambda Q_{t\wedge \tau}} \ind_E} \le \pa{1 + \lambda^2 \frac{C(1 + C_qx)\log n}{n}}^t \le n^{\lambda^2 C(1 + C_qx)} \le n^{\lambda^2 Cx}\,,
\end{equation*}
where the last inequality follows for some new constant $C$, since $C_q$ is an explicit universal constant independent of $x$, and $x$ is assumed to be sufficiently large. Finally, fix $\lambda = \min\cb{ C_q/(2C),~1/2}$, so that $\lambda \in (0,1)$ as promised. Then
\begin{align*}
    \PP{|Q_{t\wedge \tau}| > C_q x \log(n)} \le n^{\lambda^2 Cx} n^{-\lambda C_q x} + \PP{E^c} &\le \min\cb{n^{-C_q^2x/(2C)},~ n^{-C_qx/2}} + n^{-C_E} \\
    &=: n^{-cx} \,.
\end{align*}
The last line follows by taking $C_E$ arbitrarily large. 
\end{proof}

\begin{proof}[Proof of \cref{lemma:tau-improved}] By union bound and the tower property, 
\begin{align*}
    \PP{\tau \le \alpha_0 n} &= \sum_{t = 1}^{\alpha_0n} \PP{\tau = t} \\
    &\le \sum_{t = 1}^{\alpha_0n} \PP{\tau_S = t,~\tau_Q > t} + \PP{\tau_Q = t,~ \tau > t - 1} \\
    &\le n^{-cx} + \sum_{t = 1}^{\alpha_0n} \PP{\tau_S = t,~\tau_Q > t}
\end{align*}
The last line follows by \cref{lemma:Q-improved}. It then suffices to show each summand is bounded above by $n^{-cx}$. Formally, our goal is:
\begin{equation}\label{eq:tau-goal}
    \PP{\tau_S = t,~\tau_Q > t} \le n^{-cx},\quad \forall t \, \le \alpha_0 n\,.
\end{equation}

By \cref{lemma:free-energy}, there exists some constant $c_f > 0$ so that for all $n$ sufficiently large,
\begin{equation}
    \max_{\beta \in [0, 1/2 - \lambda n^{-1/2}]} F_{K,\,t/n,\,n}(\beta) \le -c_f \min \cb{\lambda^2,\, \E{|Z_{K,\,t/n}}} \,.
\end{equation}
As a convenient short-hand, we write
\begin{equation*}
    R_t(\lambda) \defeq  c_f(\lambda^2 \wedge L_t), \quad L_t \defeq  \log\pa{ \E{|Z_{K,t/n}} } \,.
\end{equation*}
Then, for two vectors $\sigma_t^{(1)}$ and $\sigma_t^{(2)}$ drawn uniformly and independently from $S_t$, it can easily be extracted from the proof of \cref{thm:2mm} that:
\begin{equation*}
    \E{\ba{\cb{ (\sigma_t^{(1)}, \sigma_t^{(2)}) \in S_t:~ \ba{  \innerprod{\sigma_t^{(1)}}{\sigma_t^{(2)}} } \ge \lambda\sqrt{n} } }} \le  \cE{-R_t(\lambda)} \E{|S_t|^2} \,.
\end{equation*}
In particular, since we showed in \cref{thm:2mm} that there exists a positive constant $C$ so that $\E{|S_t|^2} \le C\E{|S_t|}^2$ uniformly for all $t$, 
\begin{align}
    \E{ \PP[t]{\ba{  \innerprod{\sigma_t^{(1)}}{\sigma_t^{(2)}} } \ge \lambda\sqrt{n} } \frac{|S_t|^2}{\E{|S_t|}^2} } \le  C \cE{-R_t(\lambda)} \,. \label{eq:E-pairs}
\end{align}
We would now like to show an exponential tail bound on the overlap of a pair of solutions. Let us give this a short-hand name for convenience:
\[
        \mathcal{O}^t \defeq  \frac{1}{\sqrt n} \ba{  \innerprod{\sigma_t^{(1)}}{\sigma_t^{(2)}} } \,.
\]
Recall that by definition $\cE{Q_t} = |S_t|/\E{|S_t|}$. Then,
\begin{align*}
 &\E{\cE{\frac{R_t(\mathcal{O}^t)}{2} + 2Q_t}} \\
 &= \E{ \cE{\frac{1}{2}\pa{\frac{c_f\innerprod{\sigma_t^{(1)}}{\sigma_t^{(2)}}^2} {n} \wedge L_t} + 2Q_t} } \\
 &= \E{\int_0^{\infty} \PP[t]{ \cE{ \frac{1}{2}\pa{\frac{c_f\innerprod{\sigma_t^{(1)}}{\sigma_t^{(2)}}^2} {n} \wedge L_t}  } \ge y }dy \frac{|S_t|^2}{\E{|S_t|}^2}} \\
 &= \E{\int_0^{\cE{L_t/2}} \PP[t]{ \cE{ \frac{c_f}{2n} \innerprod{\sigma_t^{(1)}}{\sigma_t^{(2)}}^2} \ge y }dy \frac{|S_t|^2}{\E{|S_t|}^2}} \,.
\end{align*}
From the second to the third line, we have conditioned on $S_t$. Applying the change of variables $z \defeq  \sqrt{2\log(y)/c_f}$ and then \cref{eq:E-pairs},
\begin{align}
    \E{\cE{\frac{R_t(\mathcal{O}^t)}{2} + 2Q_t}} &= \int_0^{\sqrt{L_t/c_f }}c_f ze^{c_f z^2/2} \E{\PP[t]{ \ba{\innerprod{\sigma_t^{(1)}}{\sigma_t^{(2)}}} \ge z \sqrt{n}}  \frac{|S_t|^2}{\E{|S_t|}^2}  } dz \nonumber \\
    &\le C\int_0^{\sqrt{L_t/c}}c_f ze^{-c_f z^2/2} dz \nonumber\\
    &\le C \,. \label{eq:mgf-overlap}
\end{align}
We would like to apply Markov's inequality using \cref{eq:mgf-overlap}. Recall \cref{eq:C_q-constraint}, namely that there exists some universal constant $c_K >0$ with
\[
    L_t \ge c_K x \log (n), \quad \forall t \le \alpha_0 n\,.
\]
We have not yet fixed $C_q$ up to this point; here we will need $C_q \ll C$, e.g. $C_q  \le C/10$. For the proof of \cref{thm:sharp-transition}, we also would like $C_q < c_K$. Thus, set: 
\begin{equation}\label{eq:C_q-def}
    C_q \defeq  \min\cb{ C/10,~c_K/2 }\,.
\end{equation}
In order to establish \cref{eq:tau-goal}, note that if $\cb{\tau_Q > t}$, we have 
\begin{equation}\label{eq:Q-L_t}
    |Q_t|< .1L_t
\end{equation}
We also have yet to fix $\creg$, one of the two parameters in the definition of regular. Set 
\begin{equation}\label{eq:C_r-def}
    \creg \defeq  2C/c_f \vee 1 \,.
\end{equation} 
(We only take the maximum with one for notation convenience later). 
On the event $\cb{\tau_Q > t}$, this choice yields
\begin{align}
     R_t\pa{\sqrt{\creg x  \log (n)+ |Q_t|}} + 2Q_t &> \min\cb{ (c_f\creg-2C_q)x\log (n),~ .9 c_f L_t } \nonumber \\
     &\ge .9 R_t\pa{\sqrt{\creg x  \log (n)}} \,. \label{eq:Q<L}
\end{align}
Since $R_t(\lambda)$ is weakly increasing in $\lambda$, we have:
\begin{align}
        &\PP{\ba{\innerprod{\sigma_t^{(1)}}{\sigma_t^{(2)}}} > \sqrt{n(\creg x  \log (n)+ |Q_t|)} ,~ \tau_Q > t} \nonumber \\ &\le \PP{\cE{ R_t(\mathcal{O}^t) + 2Q_t } > \cE{  R_t\pa{ \sqrt{\creg x  \log (n)+ |Q_t|} } + 2Q_t }, ~ \tau_Q > t }\,.  \label{eq:overlap-markov} 
\end{align}
By tower property, and then applying Markov's inequality to \cref{eq:overlap-markov} via \cref{eq:mgf-overlap} and \cref{eq:Q<L}, 
\begin{align*}
    &\E{\ind_{\tau_Q > t}\PP[t]{\ba{\innerprod{\sigma_t^{(1)}}{\sigma_t^{(2)}}} > \sqrt{n(\creg x  \log (n)+ |Q_t|)}  }} \\
    &= \PP{\ba{\innerprod{\sigma_t^{(1)}}{\sigma_t^{(2)}}} > \sqrt{n(\creg x  \log (n)+ |Q_t|)} ,~\tau_Q > t} \\
    &\le  C e^{-.9R_t\pa{\sqrt{\creg x  \log (n)+ |Q_t|}}} \,.
\end{align*}
By a final application of Markov's inequality,
\begin{align}
    &\PP{  \PP[t]{\ba{\innerprod{\sigma_t^{(1)}}{\sigma_t^{(2)}}} >  \sqrt{n(\creg x  \log (n)+ |Q_t|)} } > e^{-R_t\pa{\sqrt{C_rx  \log (n)}}/2} ,~  \tau_Q > t} \\
    &\le \cO{e^{-\frac{2}{5}R_t\pa{\sqrt{C_rx  \log (n)}}}} \nonumber \\
    &\le  C \pa{ n^{-cx} + e^{-c L_t}} \,. \label{eq:reg-mean}
\end{align}
In conclusion, since $\E{|Z_{k,t}|} \ge Cx\log(n)$ for all $t \le \alpha_0n$, we may simply upper bound $e^{-cL_t}$ by $n^{-cx}$ and also set the constant $c_r$ from the definition of $S_t$ being regular to
\begin{equation} \label{eq:c_r-def}
    c_r \defeq  c_fC_r/2\,.
\end{equation}
With this notation, \cref{eq:reg-mean} implies our goal \cref{eq:tau-goal}: since $C_r \ge 1$,
\begin{align*}
    &\PP{\tau_S = t,\, \tau_Q > t} \\&:= \PP{  \PP[t]{\ba{\innerprod{\sigma_t^{(1)}}{\sigma_t^{(2)}}} >  \sqrt{\creg n(x  \log (n)+ |Q_t|)} } > n^{-c_rx} ,~  \tau_Q > t}  \\
    &\le \PP{  \PP[t]{\ba{\innerprod{\sigma_t^{(1)}}{\sigma_t^{(2)}}} >  \sqrt{n(\creg x  \log (n)+ |Q_t|)} } > e^{-R_t\pa{\sqrt{C_qx  \log (n)}}/2} ,~  \tau_Q > t}  \\
    &\le n^{-cx}\,.
\end{align*}
Taking $x$ sufficiently large immediately yields the lemma .
\end{proof}

\section{Lower Bound on Fluctuations}\label{section:LB}

\begin{lemma}\label{lemma:anticonc}
Fix positive $\alpha$ and let $\kc := \kc(\alpha)$, $Z_{K} := Z_{K,\alpha}$. There exists a positive constant $C$ such that for $K$ with $0 < K \le 2 K_c$, $n$ sufficiently large, and positive $\eps > 0$ sufficiently small: 
\begin{equation}\label{eq:anticonc}
    \PP{K - \frac{\eps}{n}  \le \disc(A) \le K } < C \eps~ \E{|Z_{K } |}  \,.
\end{equation}
\end{lemma}
As an easily corollary, we obtain a lower bound on the variance of $\disc(A)$, completing our main result \cref{thm:main}. In fact, we obtain a lower bound on fluctuations, which is strictly stronger than a bound on the variance. 

\begin{definition}[Fluctuations]\label{def:fluct-LB} We say that a sequence of random variables $X_n$ has fluctuations at least of order $\sigma_n$ if there exist constants $\del > 0$ and $C>0$ such that, for $n$ sufficiently large and any $|a-b| > C\sigma_n$, 
\[
    \PP{a < X_n < b} < 1 - \del \,.
\]
\end{definition}

\begin{corollary}\label{cor:lb}
    The fluctuations of $\disc(A)$ are at least  order $1/n$.  
\end{corollary}

\begin{proof}[Proof of \cref{cor:lb}]
By \cref{thm:2mm}, we know that for some $\delta > 0$
\[
    \PP{\disc(A) \le K_c(\alpha)} > \delta > 0\,.
\]
So, for any $a$ and $b$ possibly functions of $n$, with $K_c \le a \le b$, we have uniformly
\[
    \PP{\disc(A) \in [a,b]} < 1-\delta\,.
\]
All that remains is the case that $a \defeq  a(n)$ and $b\defeq b(n)$ with $a < b$ and $a \le K_c$. Assume $|a-b| \le \eps/n$ for some sufficiently small constant $\eps > 0$ we choose shortly. By definition of $K_c$ as well as the expansion \cref{eq:Z-expansion-K}, we have for some $C>0$,
\[
    \E{|Z_{b}|} \le Ce^{Cn(b-\kc)_+} \le C \cE{C\eps}\,.
\]
By \cref{eq:anticonc}, for $\eps$ sufficiently small,
\[
    \PP{\disc(A) \in [a,b]} < C\eps \cE{C\eps} \le \frac{1}{2}\,.
\]
\end{proof}

\begin{proof}[Proof of \cref{lemma:anticonc}]
Let $\Xi$ be the set of solutions with discrepancy inside an $\eps/n$ window around $K$, namely
\[
    \Xi \defeq  \Xi_{K,\eps} \defeq  \pa{Z_{K-\eps/n}}^c \cap Z_{K}\,.
\]
We will show $\Xi$ is empty with positive probability by Markov's inequality. Let $T_i(x)$ denote the event that row $i$ is \textit{tight}, namely
\[
    T_i \defeq  \{|(Ax)_i| \in I_{K,\eps}\}, \quad I_{K,\eps} \defeq  [K-\eps/n,~ K]\,.
\]
Correspondingly, define the number of tight rows $T$ by 
\[
    T(x) \defeq  \sum_{i=1}^{\alpha n} \ind(T_i(x))\,.
\]
Let $x$ be an arbitrary vector in $\{\pm 1\}^n$. Conditioning on $T$ and using the independence of the rows of $A$, we have for some sufficiently large  $C_0 \defeq  C_0(\alpha) > 0$ that
\begin{align*}
    \E{|\Xi|} &= 2^n \sum_{t = 0}^{\alpha n} \PP{x \in \Xi| T(x) = t} \PP{T(x) = t} \\
    &\le 2^n \sum_{t = 1}^{\alpha n}  \binom{{\alpha n}}{t}\PP{(Ax)_1 \in I_{K,\eps}}^t  \PP{\left|(Ax)_1\right| < K}^{{\alpha n}-t} \\
    &\le 2^n \sum_{t = 1}^{\alpha n} \binom{{\alpha n}}{t}\pa{\frac{C_0\eps}{n} p_K}^t p_K^{{\alpha n}-t}\,.
\end{align*}
where, as usual, $p_K = \PP{|Z| \le K}$ for $Z$ a standard Gaussian. Note that we are able to let $C_0$ only depend on $\alpha$ because $K$ is bounded above by a constant, namely $2K_c$. Thus, by the binomial theorem, for some other sufficiently large $C \defeq  C(\alpha) > 0$, 
\begin{align*}
    \E{|\Xi|} &\le (2^{n} p_K^{\alpha n}) \sum_{t = 1}^{\alpha n} \binom{{\alpha n}}{t}\pa{\frac{C_0\eps}{n}}^t 1^{{\alpha n}-t} = \E{|Z_{K}|}\pa{\pa{1 + \frac{C_0\eps}{n}}^{\alpha n} - 1} \le \E{|Z_{K}|} C\eps\,.
\end{align*}
Note that $|\Xi|$ is a non-negative integer, so the events $\cb{|\Xi| \le 1/2}$ and $\cb{|\Xi| = 0}$ are almost surely equal. Thus, Markov's inequality applied to $|\Xi|$ immediately yields the theorem. 
\end{proof}

\section*{Acknowledgment}
We thank Jonathan Niles-Weed for invaluable conversations and feedback at all stages of this manuscript. We thank Will Perkins for insightful feedback on a previous draft. We are also grateful to the anonymous reviewers. This work was funded in part by an NYU MacCracken fellowship, NSF Graduate Research Fellowship Program grant DGE-1839302, and NSF grant DMS-2015291.

\providecommand{\bysame}{\leavevmode\hbox to3em{\hrulefill}\thinspace}
\providecommand{\MR}{\relax\ifhmode\unskip\space\fi MR }
\providecommand{\MRhref}[2]{%
  \href{http://www.ams.org/mathscinet-getitem?mr=#1}{#2}
}
\providecommand{\href}[2]{#2}

\section*{Appendix}
\begin{proof}[Proof of \cref{lemma:free-energy}]
All claims follow directly from a combination of lemmas in \cite{ALS1} and \cite{APZ}. The goal of this proof is simply to organize their contributions. We refer the reader to these two papers for details. \\

\textbf{Claims 1 and 2.} The first claim is an easy computation. The second claim is proven in \cite{APZ} (page 7). Indeed, they show that $F''(1/2) < 0$ for any $K > 0$ if $\alpha \le \ac$. Fixing $K$, we thus have by monotonicity in $\alpha$ of \cref{eq:F''(1/2)}, a uniform upper bound on $F''(1/2)$. \\

\textbf{Claim 3.} Simply compute $F'(0)$ (see e.g. Eq. 12 of \cite{APZ}) and note
\[
    \lim_{\beta \to 0^+} F'(\beta) = -\infty.
\]
Since $F$ is differentiable on $(0,1/2]$ and continuous at $0$, there is then some positive radius around 0 for which $F$ is decreasing. \\

\textbf{Claim 4.} It will suffice to organize the results developed in \cite{ALS1} for the proof of their Lemma 3.4. They treat separately the three cases of: $K > 4$, $K \in [.1, 4]$, and $K \in [0, .1)$ respectively in their Lemmas 4.6; 4.10 and 4.11; and 4.7. We summarize these results. For $K > 4$, they show that $F'$ transitions from negative to positive as $\beta$ from $0$ to $1/2$, with no other sign changes. (We communicate their proof of this below). It then follows for $K > 4$ and $0 \le a \le b \le 1/2$ that
\begin{equation*}
        \max_{\beta \in [a, b]} F_{\ac}(\beta) \le \max \{F_{\ac}(a),F_{\ac}(b)\}\,.
\end{equation*}
In the remaining two cases, while conjecturally the same picture holds, the actual picture that is manageable to prove is slightly different. It is established in \cite{ALS1} only that there are some constants $0 < b_1 \le  b_2 < 1/2$ (depending only on $K$) so that $F$ is decreasing on $(0,b_1)$, increasing on $(b_2,1/2)$, and $F_{\ac}(\beta) < F_{\ac}(1/2) - \eps$ for some $\eps \defeq  \eps(K) > 0$ in between. This will yield \cref{eq:F-max}, which is enough for our purposes. We will reproduce most of the proof for $K>4$ to make this more concrete, since this case has the cleanest analysis. Then, we will remark on how $b_1$ and $b_2$ are chosen for the remaining cases and try to highlight which parts are easy and difficult for the three cases. We leave many details of this summary---especially the actual implementation of the grid search---for the reader to find in \cite{ALS1}. We also emphasize again that the following arguments and ideas we attempt to outline are not our own, but rather completely from \cite{ALS1}. \\

Fix $K$ and $\alpha = \ac(K)$, and denote for short-hand $F \defeq  F_{K,\ac,n}$. By an easy computation,
\begin{equation}\label{eq:F'}
F'(\beta) = -\log(\beta) + \log(1-\beta) + \ac  \frac{-\cE{-\frac{K^2}{2(1-\beta)}} + \cE{- \frac{K^2}{2\beta}}}{\pi q_k(\beta) \sqrt{\beta(1-\beta)}}\,.
\end{equation}
Rearranging this equation, define another quantity $\cL$ that we will aim to bound:
\begin{equation}\label{eq:L-def}
    \cL(\beta) \defeq  \frac{\ac}{\pi} \frac{\cE{-\frac{K^2}{2(1-\beta)}} - \cE{- \frac{K^2}{2\beta}}}{ \sqrt{\beta(1-\beta)} (\log(1-\beta) - \log(\beta))}\,.
\end{equation}
Note that $F'(\beta) > 0$ is equivalent to $\cL(\beta) < q(\beta)$. 

\textbf{Case 1: $4 < K$}. Consider $\beta \in[.2,.5]$. It is easy to check that
\[
    q_K(\beta) \ge q_{4}(.2) > .9\,,
\]
so our goal is to show $\cL(\beta) < .9$ for all $\beta \in [.2,~.5]$.
We have
\begin{equation}\label{eq:K>4-frac}
    \frac{\log(1-\beta) - \log(\beta)}{1/2 - \beta} \ge 4\,,
\end{equation}
and 
\begin{equation}\label{eq:K>4-exp''}
    \frac{d^2}{d\beta^2}\cE{-\frac{K^2}{2\beta}} > 0\,.
\end{equation}
Combined with Taylor's expansion, \cref{eq:K>4-exp''} yields
\begin{equation}\label{eq:K>4-exp}
    \frac{\cE{-\frac{K^2}{2(1-\beta)}} - \cE{- \frac{K^2}{2\beta}}}{-(\beta-1/2)} \le \frac{K^2\cE{-\frac{K^2}{2(1-\beta)}} }{(1-\beta)^2}\,.
\end{equation}
By \cref{eq:K>4-frac} and \cref{eq:K>4-exp}, we simplify the desired inequality:
\begin{align*}
    \cL(\beta) &\le \frac{\ac}{\pi} \frac{\cE{-\frac{K^2}{2(1-\beta)}} - \cE{- \frac{K^2}{2\beta}}}{ \sqrt{\beta(1-\beta)} (1/2-\beta)} \\
    &\le \frac{\ac}{\pi} \frac{K^2\cE{-\frac{K^2}{2(1-\beta)}} } { \sqrt{\beta(1-\beta)} (1-\beta)^2}\,.
\end{align*}
Next, using standard asymptotics for the error function, 
\[
    1 - p_K \ge \cE{-K^2/2}\sqrt{\frac{2}{\pi}} \pa{\frac{1}{K}-\frac{1}{K^3}}\,,
\]
so we have
\begin{equation}
    \ac \defeq  -\frac{\log(2)}{\log(p)} \le \frac{\log 2}{1- p_K} \le \frac{16}{15}\log(2) \sqrt{\frac{\pi}{2}}\cE{K^2/2}\,.
\end{equation}
These inequalities suffice for the sub-case that $x \in [.35,~.5]$. In total:
\begin{align*}
    \cL(\beta) &\le \frac{1}{4}\log(2) \sqrt{\frac{\pi}{2}}\cE{K^2/2} \frac{16}{15} \frac{K^2\cE{-\frac{K^2}{2(1-\beta)}} } { \sqrt{\beta(1-\beta)} (1-\beta)^2} \frac{1}{\sqrt{.2 \cdot .8}} \\
    &\le \cE{K^2\pa{\frac{1}{2} - \frac{1}{2(1-\beta)}} } \frac{K^3}{(1-x)^2}.19 \\
    &< .9 \,.
\end{align*}
In the sub-case that $x \in [.2,~.35]$, replace \cref{eq:K>4-exp} with the trivial inequality
\begin{equation*}
    \frac{\cE{-\frac{K^2}{2(1-\beta)}} - \cE{- \frac{K^2}{2\beta}}}{-(\beta-1/2)} \le \frac{\cE{-\frac{K^2}{2(1-\beta)}}}{-(\beta-1/2)}\,.
\end{equation*}
Then an identical computation easily yields $\cL(\beta) < .9$. So, for all $\beta \in [.2,.5]$, we have $F'(\beta) > 0$. 

Suppose that $\cL'(\beta) < q'(\beta)$ for $\beta \le .2$. Then $L$ and $q$ can cross at most once. Note $F' < 0$ for some ball of positive radius around $\beta = 0$ by claim three, so $L > q$ for all $\beta$ sufficiently small. We have also already shown that $L < q$ for $\beta > .2$. Thus, $\cL$ and $q$ cross exactly once, and we have the desired picture: $F$ is decreasing and then increasing, with no other changes. In order to establish $\cL'(\beta) < q'(\beta)$ for $\beta \le .2$, one can differentiate the definition of $L$, and then apply standard tail-bounds on the error function. \\

\textbf{Case 2: $.1 \le K \le 4$}. The strategy is almost the same. We know that $F$ is decreasing in some neighborhood of 0, so a natural first step is to, again, show that $F$ is increasing in some neighborhood of $1/2$. In the previous case, we were able to establish $F$ is increasing in $[.2,~.5]$. Here, \cite{ALS1} were only able to show this for $\beta \in [.3,~.5]$. The proof is quite similar, except that, roughly speaking, while bounds derived from first-order Taylor expansion sufficed for $K>4$, second-order expansions are needed for $K \in [.1,~.4]$. Next, we again try and make Claim 3 (namely that $F$ is decreasing near 0) quantitative. By looking very close to 0, namely $\beta \in [0,~.005]$, we can actually prove a stronger result with easier computations. Since trivially $p \ge q$, it of course suffices to  show $\cL > p_K$ for $\beta$ near 0. The strategy is to then show that $\cL(.005) > p_K$ and $\cL' < 0$ for $\beta \in [0,.005]$. These are simple computations. Finally, in the remaining region $\beta \in [.005,~.3]$, we would like to check that $F(\beta) - F(1/2) < - \eps$ for some $\eps \defeq  \eps(K) > 0$. This is done in two steps: first, it is relatively straight-forward to check that the derivative of $F_K(\beta) - F_K(1/2)$ with respect to $\beta$ and the derivative with respect to $K$ are both bounded in absolute value by e.g. $6$. Second, do a computerized grid search over the region 
\[
    \cb{(\beta,K):~ \beta \in [.005,~.3], ~ K \in [.1, 4]}\,.
\]
Since we have a bound on how fast $F_K(\beta)-F_K(1/2)$ can change, a step-size can be picked appropriately to obtain a provably tolerable error and conclude. In summary, taking $b_1$ and $b_2$ as $.005$ and $.3$ respectively, the picture we have established is that $F$ is decreasing on $[0,b_1]$; increasing on $[b_2,1/2]$; and strictly negative between. \\

\textbf{Case 3: $0 \le K \le .1$}. Here, $K$ being small will actually make the analysis quite tractable. Indeed, we can for example get decent control over $p$ and $q$, which will be used repeatedly. Using a zero-order and first-order Taylor expansion of $\cE{-x^2/2}$ for $x \in [-K,K]$ for the upper and lower bound respectively, we have trivially
\begin{equation}\label{eq:p-bd-ksmall}
    \frac{.992K}{\sqrt{2\pi}} \le p_K \le \frac{2K}{\sqrt{2\pi}}\,.
\end{equation}
One can also bound $q_K(\beta)$ in a similarly crude way:
\begin{equation}\label{eq:q-bd-ksmall}
    \frac{K^2}{\pi \sqrt{\beta(1-\beta)}}\cE{-\frac{K^2}{2\beta}} \le q_K(\beta) \le \frac{K^2}{\pi \sqrt{\beta(1-\beta)}}\,.
\end{equation}

Note that since $K$ is small, the gap between the upper and lower bounds here is very small if $\beta$ is not too small as well. With these bounds, we are now ready to prove the desired picture. First, we claim for $\beta \in [.27,~.5]$ that $\cL(\beta) \le q(\beta)$. The main difference from the first case ($K > 4$) is that now
\begin{equation}
    \frac{d^2}{d\beta^2}\cE{-\frac{K^2}{2\beta}} < 0\,,
\end{equation}
so we replace \cref{eq:K>4-exp} by
\begin{equation*}
    \frac{\cE{-\frac{K^2}{2(1-\beta)}} - \cE{- \frac{K^2}{2\beta}}}{-(\beta-1/2)} \le \frac{K^2\cE{-\frac{K^2}{2\beta}} }{\beta^2}\,.
\end{equation*}
Next, for $\beta \in (0,K^2/2)$, we need to verify that $\cL(\beta) > q(\beta)$. Indeed, since $\sqrt{\beta}\log(\beta)$ is decreasing for $\beta \le .005$ (and here we are only considering $\beta \le K^2/2 \le .005$),
\begin{align*}
    -\log(p)\cL(\beta) &\defeq  \frac{\log 2}{\pi}\frac{\cE{-\frac{K^2}{2(1-\beta)}} - \cE{- \frac{K^2}{2\beta}}}{ \sqrt{\beta(1-\beta)} (\log(1-\beta) - \log(\beta))} \\
    &\ge \frac{\log 2}{\pi}\frac{1/2}{ \sqrt{\beta} ( - \log(\beta))} \\
    &\ge .29\,,
\end{align*}
where the last line follows from monotonicity of the penultimate formula in $\beta$, and plugging in $\beta = .005$. From here, it remains to check that $-\log(p_K)q(\beta) > .29$, so that $F$ is decreasing on $\beta \in [0,~K^2/2]$. This is readily verified by applying \cref{eq:q-bd-ksmall} to bound $q$; and then crudely bounding $-\log p_K \ge -\log p_{.1}$ using \cref{eq:p-bd-ksmall}. \\

Carrying out the Taylor expansions used in \cref{eq:q-bd-ksmall} and \cref{eq:p-bd-ksmall} to another term, the same argument shows that $F$ is actually decreasing on $[0,~K/12]$ and then increasing on $[.04,~ 1/2]$. Finally, for the remaining region of $[K/12,~.04]$, we can show directly that $F(\beta) < F(1/2)-\eps$ for some $\eps \defeq  \eps(K) > 0$. Indeed, recalling that $F_{\ac(K)}(1/2) = \log(2) + 2\ac(K) \log(p)$, the desired claim follows quickly from the inequality 
\[
    \log \pa{\frac{q(\beta)}{p^2}} < \log\pa{\frac{1}{.99^22\sqrt{\beta}}} < \log \pa{ \frac{2}{\sqrt{K}}}\,.
\]
In summary, for $b_0$ and $b_1$ equal to $K/12$ and $.04$ respectively, we have again established the picture that $F$ is decreasing on $[0,b_1]$; increasing on $[b_2,1/2]$; and strictly negative between. This concludes the analysis of the final case $K \le .01$, and thus the outline of claim four. 
\end{proof}


\begin{thebibliography}{10}

\bibitem{ALS2}
Emmanuel Abbe, Shuangping Li, and Allan Sly, \emph{Binary perceptron: efficient
  algorithms can find solutions in a rare well-connected cluster}, S{TOC}'21---{P}roceedings of
  the 2021 {ACM} {S}ymposium on {T}heory of {C}omputing, 2021.

\bibitem{ALS1}
Emmanuel Abbe, Shuangping Li, and Allan Sly, \emph{Proof of the contiguity
  conjecture and lognormal limit for the symmetric perceptron}, F{OCS}'21---{IEEE} 62nd {A}nnual {S}ymposium on {F}oundations of {C}omputer {S}cience, 2021.

\bibitem{kcolor}
Dimitris Achlioptas and Ehud Friedgut, \emph{A sharp threshold for
  {$k$}-colorability}, Random Struct. Algorithms (1999), no.~1,
  63--70. \MR{1662274}

\bibitem{aiz}
M.~Aizenman, J.~L. Lebowitz, and D.~Ruelle, \emph{Some rigorous results on the
  {S}herrington-{K}irkpatrick spin glass model}, Comm. Math. Phys. \textbf{112}
  (1987), no.~1, 3--20. \MR{904135}

\bibitem{dja-jnw}
Dylan~J Altschuler and Jonathan Niles-Weed, \emph{The discrepancy of random
  rectangular matrices}, Random Struct. Algorithms (2021).

\bibitem{APZ}
Benjamin Aubin, Will Perkins, and Lenka Zdeborov{\'{a}}, \emph{Storage capacity
  in symmetric binary perceptrons}, J. Phys. A \textbf{52} (2019), no.~29,
  294003, 32. \MR{3983947}

\bibitem{banerjee}
Debapratim Banerjee, \emph{Contiguity and non-reconstruction results for
  planted partition models: the dense case}, Electron. J. Probab. \textbf{23}
  (2018), Paper No. 18, 28. \MR{3771755}

\bibitem{bansal-meka}
Nikhil Bansal and Raghu Meka, \emph{On the discrepancy of random low degree set
  systems}, Random Struct. Algorithms \textbf{57} (2020), no.~3, 695--705.
  \MR{4144080}

\bibitem{2SAT}
B\'{e}la Bollob\'{a}s, Christian Borgs, Jennifer~T. Chayes, Jeong~Han Kim, and
  David~B. Wilson, \emph{The scaling window of the 2-{SAT} transition}, Random
  Struct. Algorithms \textbf{18} (2001), no.~3, 201--256. \MR{1824274}

\bibitem{small-densities}
Erwin Bolthausen, Shuta Nakajima, Nike Sun, and Changji Xu, \emph{Gardner formula for Ising perceptron models at small densities}, Conference on Learning Theory, 2022,
  pp.~1787--1911.

\bibitem{int-feas}
Karthekeyan Chandrasekaran and Santosh~S. Vempala, \emph{Integer feasibility of
  random polytopes}, I{TCS}'14---{P}roceedings of the 2014 {C}onference on
  {I}nnovations in {T}heoretical {C}omputer {S}cience, ACM, New York, 2014,
  pp.~449--458. \MR{3359497}

\bibitem{DSS-KSAT}
Jian Ding, Allan Sly, and Nike Sun, \emph{Proof of the satisfiability
  conjecture for large {$k$} [extended abstract]}, S{TOC}'15---{P}roceedings of
  the 2015 {ACM} {S}ymposium on {T}heory of {C}omputing, ACM, New York, 2015,
  pp.~59--68. \MR{3388183}

\bibitem{maxindset}
\bysame, \emph{Maximum independent sets on random regular graphs}, Acta Math.
  \textbf{217} (2016), no.~2, 263--340. \MR{3689942}

\bibitem{NAE-SAT}
\bysame, \emph{Satisfiability threshold for random regular {NAE}-{SAT}}, Comm.
  Math. Phys. \textbf{341} (2016), no.~2, 435--489. \MR{3440193}

\bibitem{SAT}
\bysame, \emph{Proof of the satisfiability conjecture for large {$k$}}, Ann. of
  Math. (2) \textbf{196} (2022), no.~1, 1--388. \MR{4429261}

\bibitem{dingsun-perceptron}
Jian Ding and Nike Sun, \emph{Capacity lower bound for the {I}sing perceptron},
  S{TOC}'19---{P}roceedings of the 51st {A}nnual {ACM} {SIGACT} {S}ymposium on
  {T}heory of {C}omputing, ACM, New York, 2019, pp.~816--827. \MR{4003386}

\bibitem{ezra-lovett}
Esther Ezra and Shachar Lovett, \emph{On the {B}eck-{F}iala conjecture for
  random set systems}, Random Struct. Algorithms \textbf{54} (2019), no.~4,
  665--675. \MR{3957362}

\bibitem{franks2020discrepancy}
Cole Franks and Michael Saks, \emph{On the discrepancy of random matrices with
  many columns}, Random Struct. Algorithms \textbf{57} (2020), no.~1,
  64--96. \MR{4120593}

\bibitem{fried}
Ehud Friedgut, \emph{Sharp thresholds of graph properties, and the {$k$}-sat
  problem}, J. Amer. Math. Soc. \textbf{12} (1999), no.~4, 1017--1054, With an
  appendix by Jean Bourgain. \MR{1678031}

\bibitem{gamarnik-sbp}
David Gamarnik, Eren~C K{\i}z{\i}lda{\u{g}}, Will Perkins, and Changji Xu,
  \emph{Algorithms and barriers in the symmetric binary perceptron model}, F{OCS}'22---{IEEE} 62nd {A}nnual {S}ymposium on {F}oundations of {C}omputer {S}cience, 2022.

\bibitem{gd1}
E.~Gardner and B.~Derrida, \emph{Optimal storage properties of neural network
  models}, J. Phys. A \textbf{21} (1988), no.~1, 271--284. \MR{939731}

\bibitem{hatami}
Hamed Hatami, \emph{A structure theorem for {B}oolean functions with small
  total influences}, Ann. of Math. (2) \textbf{176} (2012), no.~1, 509--533.
  \MR{2925389}

\bibitem{hoberg-rothvoss}
Rebecca Hoberg and Thomas Rothvoss, \emph{A {F}ourier-analytic approach for the
  discrepancy of random set systems}, Proceedings of the {T}hirtieth {A}nnual
  {ACM}-{SIAM} {S}ymposium on {D}iscrete {A}lgorithms, SIAM,
  2019. \MR{3909625}

\bibitem{dreg}
Jiaoyang Huang, \emph{Invertibility of adjacency matrices for random
  {$d$}-regular graphs}, Duke Math. J. \textbf{170} (2021), no.~18, 3977--4032.
  \MR{4348232}

\bibitem{macrury2021phase}
Calum MacRury, Tomáš Masařík, Leilani Pai, and Xavier Pérez-Giménez,
  \emph{The phase transition of discrepancy in random hypergraphs}, ArXiv 2102.07342, 2021.

\bibitem{sandpile}
Andr\'{a}s M\'{e}sz\'{a}ros, \emph{The distribution of sandpile groups of
  random regular graphs}, Trans. Amer. Math. Soc. \textbf{373} (2020), no.~9,
  6529--6594. \MR{4155185}

\bibitem{mez}
Marc M\'{e}zard, \emph{The space of interactions in neural networks:
  {G}ardner's computation with the cavity method}, J. Phys. A \textbf{22}
  (1989), no.~12, 2181--2190, Special issue in memory of Elizabeth Gardner,
  1957--1988. \MR{1004920}

\bibitem{nakajima-sun}
Shuta Nakajima and Nike Sun, \emph{Sharp threshold sequence and universality
  for ising perceptron models}, Proceedings of the {A}nnual
  {ACM}-{SIAM} {S}ymposium on {D}iscrete {A}lgorithms, 2023.

\bibitem{PX}
Will Perkins and Changji Xu, \emph{Frozen $1$-rsb structure of the symmetric
  ising perceptron}, S{TOC}'21---{P}roceedings of the {A}nnual {ACM} {SIGACT} {S}ymposium on
  {T}heory of {C}omputing, ACM, New York, 
  2021.







\bibitem{XOR}
Boris Pittel and Gregory~B. Sorkin, \emph{The satisfiability threshold for
  {$k$}-{XORSAT}}, Combin. Probab. Comput. \textbf{25} (2016), no.~2, 236--268.
  \MR{3455676}

\bibitem{potukuchi2018discrepancy}
Aditya Potukuchi, \emph{Discrepancy in random hypergraph models}, arXiv
  preprint arXiv:1811.01491 (2018).

\bibitem{rw}
R.~W. Robinson and N.~C. Wormald, \emph{Almost all regular graphs are
  {H}amiltonian}, Random Struct. Algorithms \textbf{5} (1994), no.~2,
  363--374. \MR{1262985}

\bibitem{rb}
Frank Rosenblatt, \emph{The perceptron: a probabilistic model for information
  storage and organization in the brain.}, Psychological review \textbf{65}
  (1958), no.~6, 386.

\bibitem{RV-ssv}
Mark Rudelson and Roman Vershynin, \emph{Non-asymptotic theory of random
  matrices: extreme singular values}, Proceedings of the {I}nternational
  {C}ongress of {M}athematicians. {V}olume {III}, Hindustan Book Agency, New
  Delhi, 2010, pp.~1576--1602. \MR{2827856}

\bibitem{ss}
Ashwin Sah and Mehtaab Sawhney, \emph{Distribution of the threshold for the symmetric perceptron}, ArXiv 2301.10701, 2023.


\bibitem{corr}
Zbyn\v{e}k \v{S}id\'{a}k, \emph{On multivariate normal probabilities of rectangles: {T}heir dependence on correlations}, Ann. Math. Statist. (1968). \MR{230403}


\bibitem{spencer1985six}
Joel Spencer, \emph{Six standard deviations suffice}, Trans. Amer. Math. Soc.
  \textbf{289} (1985), no.~2, 679--706. \MR{784009}

\bibitem{talagrand1999intersecting}
Michel Talagrand, \emph{Intersecting random half cubes}, Random Structures
  Algorithms \textbf{15} (1999), no.~3-4, 436--449, Statistical physics methods
  in discrete probability, combinatorics, and theoretical computer science
  (Princeton, NJ, 1997). \MR{1716771}

\bibitem{tgd}
\bysame, \emph{Intersecting random half-spaces: toward the {G}ardner-{D}errida
  formula}, Ann. Probab. \textbf{28} (2000), no.~2, 725--758. \MR{1782273}

\bibitem{turner2020balancing}
Paxton Turner, Raghu Meka, and Philippe Rigollet, \emph{Balancing {G}aussian
  vectors in high dimension}, Conference on Learning Theory, 2020,
  pp.~3455--3486.


\bibitem{xu}
Changji Xu, \emph{Sharp threshold for the {I}sing perceptron model}, Ann.
  Probab. \textbf{49} (2021), no.~5, 2399--2415. \MR{4317708}

\end{thebibliography}
\end{document}